\documentclass[11pt, reqno]{amsart}

\usepackage[left=1.5in,right=1.5in]{geometry}
\usepackage[colorlinks=true,linkcolor=blue,citecolor=blue]{hyperref}

\usepackage{amscd,amssymb,amsmath,graphicx,verbatim,mathrsfs}
\usepackage{extarrows}
\usepackage{dcpic,mathtools,pictexwd}
\usepackage[all,cmtip]{xy}
\usepackage{amsthm}

\usepackage{tikz-cd}
\usepackage{adjustbox}

\newtheorem{thm}{Theorem}[section]
\newtheorem{lem}[thm]{Lemma}

\newtheorem{prop}[thm]{Proposition}
\newtheorem{conj}[thm]{Conjecture}

\newtheorem{thmx}{Theorem}

\theoremstyle{definition}
\newtheorem{defn}[thm]{Definition}
\newtheorem{eg}[thm]{Example}
\newtheorem{rmk}[thm]{Remark}

\numberwithin{equation}{section}

\begin{document}

\title[Relative group $C^\ast$-algebras and relative Novikov conjecture]{$K$-theory of relative group $C^*$-algebras and the relative Novikov conjecture}

\author{Jintao Deng}
\address[Jintao Deng]{Department of Mathematics, University of Waterloo}
\email{jintao.deng@uwaterloo.ca}
%\thanks{}

\author{Geng Tian}
\address[Geng Tian]{School of Mathematics, Liaoning University}
\email{gengtian.math@gmail.com}
%\thanks{}
%
\author{Zhizhang Xie}
\address[Zhizhang Xie]{ Department of Mathematics, Texas A\&M University }
\email{xie@math.tamu.edu}
\thanks{The third author is partially supported by NSF 1800737 and 1952693.}
\author{Guoliang Yu}
\address[Guoliang Yu]{ Department of
	Mathematics, Texas A\&M University}
\email{guoliangyu@math.tamu.edu}
\thanks{The fourth author is partially supported by NSF 1700021, 2000082, and the Simons Fellows Program.}

\date{}
%\onehalfspacing

\begin{abstract}
 The relative Novikov conjecture states that  the relative higher signatures of manifolds with boundary
are invariant under orientation-preserving homotopy equivalences of pairs. The relative Baum-Connes assembly encodes information about the relative higher index of elliptic operators on manifolds with boundary. In this paper, we study the relative Baum-Connes assembly map for any pair of groups and  apply it to solve the relative Novikov conjecture when the groups satisfy certain geometric conditions.
\end{abstract}

\maketitle

\tableofcontents

\section{Introduction}

A fundamental problem in  topology is the Novikov conjecture which states that the higher signatures of a closed (i.e. compact without boundary) oriented smooth manifold are invariant under orientation-preserving homotopy equivalences. The Novikov conjecture has been proved for a large class of manifolds by techniques from noncommutative geometry and geometric group theory. While the Novikov conjecture concerns with closed manifolds, there is a natural analogue, called the relative Novikov conjecture,  for compact oriented manifolds with boundary.
The relative Novikov conjecture states that  the relative higher signatures of a compact oriented smooth manifold with boundary are invariant under orientation-preserving homotopy equivalences of pairs. The purpose of this article is to develop a $C^*$-algebraic approach to the relative Novikov conjecture. In particular, we prove that the relative Novikov conjecture holds for a compact oriented smooth manifold with boundary  if the fundamental groups of the manifold and its boundary satisfy certain geometric conditions.

Suppose $M$ is a compact oriented manifold with boundary $\partial M$. Let $G=\pi_1 (\partial M)$ and $\Gamma=\pi_1 M$ denote their fundamental groups. Moreover, let $h\colon  G \to \Gamma$ be the group homomorphism induced by the inclusion $\partial M \hookrightarrow M$.  Suppose $\underline{E}G$ (resp. $\underline{E}\Gamma$) is the universal space for proper $G$ (resp. $\Gamma$) actions. Then $h$ induces a $(G, \Gamma)$-equivariant continuous map from $\underline{E}G$ to $\underline{E}\Gamma$, that is, the map commutes with the actions of $G$ and $\Gamma$.  One can  define a  relative Baum-Connes assembly map
$$
\mu_{max}: K^{G,\Gamma}_*(\underline{E}G,\underline{E}\Gamma) \to K_*(C^*_{max}(G, \Gamma)),
$$
where $K^{G,\Gamma}_*(\underline{E}G,\underline{E}\Gamma)$ is the relative $K$-homology for the pair  $(\underline{E}G,\underline{E}\Gamma)$ with respect to $h$ and $C^*_{max}(G,\Gamma)$ is the maximal relative group $C^*$-algebra of the pair of groups $(G, \Gamma)$ with respect to $h$. We show that the injectivity of the above relative Baum--Connes assembly map $\mu_{\max}$ implies the relative Novikov conjecture. In general, the injectivity of  the  relative Baum--Connes assembly map $\mu_{\max}$ remains an open question. In this paper,  we verify the injectivity of this relative Baum--Connes assembly map under certain geometric assumptions on the groups $\Gamma$ and $\ker(h)$. Here $\ker(h)=\left\{g \in G \mid  h(g)=e\right\}$, where $e\in \Gamma$ is the identity of $\Gamma$.

Before we state the main results of this paper, let us first introduce the following notion of group homomorphisms with good kernel property.
\begin{defn}\label{weakgamma}
 Let $G$ and $\Gamma$ be countable discrete groups.
 \begin{enumerate}
     \item[(1)] A homomorphism $h\colon G\to \Gamma$ has maximal good kernel property if for any subgroup $G'\subseteq G$ containing $\ker(h)$ with $[G^\prime: \ker(h)]<\infty$, the maximal Baum--Connes conjecture with coefficients holds for $G^\prime$.
     \item[(2)] A homomorphism $h:G\to \Gamma$ has reduced good kernel property if for any subgroup $G'\subseteq G$ containing $\ker(h)$ with $[G^\prime: \ker(h)]<\infty$, the reduced Baum--Connes conjecture with coefficients holds for $G^\prime$.
  \end{enumerate}
\end{defn}

\begin{eg}\leavevmode
\begin{enumerate}
    \item[(1)] If $\ker(h)$ is a-T-menable, then $h$ has both the maximal and the reduced good kernel properties.
     \item[(2)] When $\ker(h)$ is word hyperbolic in the sense of Gromov, the map $h$ has the reduced good kernel property.
\end{enumerate}
\end{eg}

%[Theorem \ref{thm-maxi-Novikov}]
To motivate one of our main theorems (Theorem \ref{thmb}), we first show  the following result.
\begin{thmx}\label{thma}
Let  $h:G\rightarrow \Gamma$ be a group homomorphism with the maximal good kernel property. Suppose that $\Gamma$ admits a coarse embedding into Hilbert space. Then the maximal relative assembly map
$$\mu_{max}:K_{*}^{G,\Gamma}(\underline{E}G,\underline{E}\Gamma){\rightarrow}  K_*(C_{max}^{*}(G,\Gamma))$$
is injective.
\end{thmx}
For example, if the kernel $\ker(h)$ is a-T-menable and $\Gamma$ admits a coarse embedding into Hilbert space, then the relative assembly map
$$
\mu_{max}: K^{G,\Gamma}_*(\underline{E}G,\underline{E}\Gamma) \to K_*(C^*_{max}(G, \Gamma))
$$
 is injective. We mention that   Y. Kubota  also proved the above maximal strong relative Novikov conjecture under slightly stronger assumptions \cite{MR4170653}.
 We thank Y. Kubota for bringing this to our attention.

There are many groups such as hyperbolic groups with property (T) that satisfy the reduced Baum-Connes conjecture with coefficients, but fail the maximal Baum-Connes conjecture with coefficients. For such groups, it is more natural to consider a reduced version of the relative Baum--Connes assembly map. However, the relative reduced group $C^\ast$-algebra for a general group homomorphism $h\colon G\to\Gamma$ is \emph{not} defined, unless one imposes strong restrictions on the kernel $\ker(h)$, which would be restrictive for some applications. To overcome this difficulty, we instead consider
the relative reduced group $C^\ast$-algebra $C^*_{red}(G, \Gamma, \mathcal M)$ with coefficients in a ${\rm II}_1$-factor $\mathcal M$, which is well-defined for an arbitrary group homomorphism $h\colon G\to \Gamma$, due  to the presence of $\mathcal M$.  The use of ${\rm II}_1$-factors is inspired by the work of Antonini, Azzali and Skandalis \cite{An-Az-Ska-Bivariant-K, An-Az-Ska-BC-conjecture-localised}. Our first main result of the paper is as follows.

 \begin{thmx}\label{thmb}
Let  $h:G\rightarrow \Gamma$ be a group homomorphism with the reduced good kernel property.  Assume that $\Gamma$ is coarsely embeddable into Hilbert space. Then the relative Baum-Connes assembly map
 $$\mu_{red}\colon  K^{G, \Gamma}_{*}(\underline{E}G, \underline{E}\Gamma, \mathcal M)\to  K_*(C^*_{red}(G, \Gamma, \mathcal  M))$$
 is injective, where $K^{G,\Gamma}_{*}(\underline{E}G,\underline{E}\Gamma, \mathcal M)$ is the relative $K$-homology with coefficients in $\mathcal  M$  and $ K_*(C^*_{red}(G, \Gamma, \mathcal M))$ is the $K$-theory of reduced relative group $C^*$-algebras with coefficients in $\mathcal M$.
\end{thmx}
As an example, when $\ker(h)$ is hyperbolic and the group $\Gamma$ admits a coarse embedding into Hilbert space, the relative Baum--Connes assembly map $\mu_{red}$ is injective.

For a given compact oriented manifold with boundary $(M, \partial M)$, let $(D_M, D_{\partial M})$ be the associated pair of signature operators.  Then the maximal relative Baum-Connes assembly map $\mu_{\max}$ maps $(D_M, D_{\partial M})$ to the maximal relative higher index
$$\mbox{Ind}_{max}(D_M, D_{\partial M}) \in K_*(C^*_{max}(G,\Gamma)).$$
Similarly, the reduced relative Baum-Connes assembly map $(D_M, D_{\partial M})$ to the reduced relative higher index
$$\mbox{Ind}_{red}(D_M, D_{\partial M}) \in K_*(C^*_{red}(G,\Gamma, \mathcal M)).$$ In order to apply the above theorems to the relative Novikov conjecture, we prove the following theorem  which states that maximal (resp. reduced) relative higher indices of signature operators are invariant under orientation-preserving homotopy equivalences of pairs.
\begin{thmx}\label{thmc}
  Let $M$ be a compact manifold with boundary $\partial M $ and $N$ a compact manifold with boundary $\partial N$. Let $G=\pi_1( \partial M)=\pi_1 (\partial N)$ and $\Gamma=\pi_1M=\pi_1N$.  Let $D_M$ and $D_N$ be the signature operators on $M$ and $N$, respectively. If there is an orientation-preserving homotopy equivalence $f\colon (M, \partial M) \to (N, \partial N)$, then
$$
f_\ast({\rm Ind}_{max}(D_M, D_{\partial M}))={\rm Ind}_{max}(D_N, D_{\partial N}) \in K_*(C^*_{max}(G, \Gamma)),
$$
and
$$
f_\ast({\rm Ind}_{red}(D_M, D_{\partial M}))={\rm Ind}_{red}(D_N, D_{\partial N}) \in K_*(C^*_{red}(G, \Gamma, \mathcal  M)),
$$
where for example ${\rm Ind}_{max}(D_M, D_{\partial M})$ (resp.  ${\rm Ind}_{red}(D_M, D_{\partial M})$) is the maximal (resp. reduced) relative higher index of the pair of signature operators $(D_M, D_{\partial M})$.
\end{thmx}

Combining Theorem \ref{thma}, Theorem \ref{thmb} with Theorem \ref{thmc}, we have the following theorem on the relative Novikov conjecture.
\begin{thmx}
  Let $(M, \partial M)$  and $(N, \partial N)$ be compact oriented smooth manifolds with boundary. Suppose   $f: (M, \partial M)\to (N, \partial N)$ is an  orientation-preserving homotopy equivalence.  Denote $G=\pi_1 (\partial M)\cong \pi_1 (\partial N)$ and $\Gamma=\pi_1M\cong \pi_1N$. Let $h\colon G\to \Gamma$ be the group homomorphism induced by the inclusion map $\partial M \hookrightarrow M$. If   the kernel of  $h\colon G \to  \Gamma$ is hyperbolic or a-T-menable, and $\Gamma$ admits a coarse embedding into Hilbert space, then the relative Novikov conjecture holds, i.e.,  the relative higher signatures of $(M, \partial M)$ and $(N, \partial N)$ are invariant under the homotopy equivalence $f$.
% \begin{enumerate}
%    \item[(1)]
%    \item[(2)] If the kernel of  $h\colon G \to  \Gamma$ is  a-T-menable, and $\Gamma$ admits a coarse embedding into Hilbert space, then the relative Novikov conjecture holds, i.e. the relative higher signatures of $(M, \partial M)$ and $(N, \partial N)$ are invariant under the homotopy equivalence  $f$.
%  \end{enumerate}
\end{thmx}

The paper is organized as follows. In Section 2, we formulate the maximal strong relative Novikov conjecture for a pair of discrete groups $(G,\Gamma)$. In section 3, we introduce the reduced strong relative Novikov conjecture and show the assembly map with a $\mbox{II}_1$-factor is an isomorphism for a pair of hyperbolic groups.  In Section 4 and 5, we prove the maximal and reduced strong relative Novikov conjecture under some geometric  assumptions on the group $G$, $\Gamma$ and the kernel of $h\colon G\to \Gamma$. In Section 6, we define the relative higher index for signature operators on manifolds with boundary,  and show that the relative higher indices of signature operators is invariant under orientation-preserving homotopy equivalences of pairs.

\section{The relative Novikov conjecture}

In this section, we shall first recall the definition of Roe algebras and localization algebras, then introduce the notions of relative Roe algebras and relative localization algebras associated with a group homomorphism $h:G \to \Gamma$. Finally, we construct the maximal relative Baum--Connes assembly maps.
%=================
%1.  chang 'firstly to first'
%2. Five Lemma Change to the five lemma.
%==============================================================
\subsection{Roe algebras and localization algebras} In this subsection, we recall the notions of Roe algebras and localization algebras for a metric space $Z$ endowed with a proper $G$-action (c.f. \cite{Willett-Yu-higher-index-book}).

Let $Z$ be a metric space with a proper $G$-action by isometries, and $A$ a $G$-$C^*$-algebra. A $G$-action on a Hausdorff space $Z$ is said to be proper if for every $x, y \in Z$ there exist neighborhood $U_x$ and $U_y$ of $x$ and $y$ respectively such that the set
$$
\left\{g \in G: g\cdot U_x \cap U_y\neq \emptyset\right\}
$$
is finite. A $G$-action is said to be cocompact if the quotient space $Z/G$ is compact.

\begin{defn} Let $H$ be a Hilbert module over the $C^*$-algebra $A$, and $\varphi: C_0(Z)\to B(H)$ a $*$-representation, where $B(H)$ is the $C^*$-algebra of all bounded (adjointable) operators on $H$. Let $T: H \to H$ be an adjointable operator.
\begin{enumerate}
    \item[(1)] The support of $T$, denoted by $\mbox{Supp}(T)$, is defined to be the complement of the set of all points $(x, y)\in Z \times Z$ for which there exists $f \in C_0(Z)$ and $g \in C_0(Z)$ such that $f\cdot T\cdot g= 0$, and $f(x)\neq 0$ and $g(y)\neq 0$;
    \item[(2)] The propagation of the operator $T$ is defined by
    $$\mbox{propagation}(T)=\sup \left\{ d(x,y): (x,y) \in \mbox{Supp}(T)\right\}.$$
    An operator $T$ is said to have finite propagation if $\mbox{propagation}(T)< \infty$;
    \item[(3)] The operator $T$ is said to be locally compact if $f\cdot T$ and $T \cdot f$ are in $K(H)$  for all $f \in C_0(Z)$, where $K(H)$ is the operator norm closure of all finite rank operators on the Hilbert module $H$.
    \item[(4)] The operator $T$ is said to be $G$-invariant if $g\cdot T=T\cdot g$ for all $g \in G$.
\end{enumerate}
\end{defn}

Let $H$ be a $G$-Hilbert module over $A$. A $*$-representation $\varphi: C_0(Z) \to B(H)$ is covariant if
$$
\varphi(\gamma f)v=(\gamma \varphi(f)\gamma^{-1})v
$$
for all $\gamma \in G$, $f \in C_0(Z)$ and $v \in H$. The triple $(C_0(Z), G, H)$ is called a covariant system.

\begin{defn}
We define the covariant system $(C_0(Z), G, H)$ to be admissible if
\begin{enumerate}
    \item[(1)] the $G$-action on $Z$ is proper and cocompact;
    \item[(2)] there exists a $G$-Hilbert  module $H_Z$ such that
    \begin{itemize}
        \item $H$ is isomorphic to $H_Z \otimes A$ as $G$-Hilbert modules over $A$;
        \item $\varphi=\varphi_0 \otimes I$ for some $G$-equivariant $*$-homomorphism $\varphi_0:C_0(Z) \to B(H_Z)$ such that $\varphi_0(f)$ is not in $K(H_Z)$ for any non-zero function $f \in C_0(Z)$ and $\varphi_0$ is non-degenerate in the sense that
        \[ \left\{\varphi_0(f)v: v \in H_Z, f \in C_0(Z)\right\}
        \] is dense in $H_Z$;
        \item for any finite subgroup $F \subseteq  G$ and any $F$-invariant Borel subset $E$ of $Z$, there is Hilbert space $H_E$ with trivial $F$-action such that $\chi_E H_Z$ and $\ell^2(F) \otimes H_E$ are isomorphic as $F$-representations.
    \end{itemize}

\end{enumerate}
\end{defn}

\begin{defn}
Let $(C_0(Z), \varphi, H)$ be an admissible system. The algebraic Roe algebra with coefficients in $A$, denoted by $\mathbb{C}[Z, A]^G$, is defined to be the algebra  of $G$-invariant locally compact operators in $B(H)$ with finite propagation. The Roe algebra $C^*(Z, A)^G$ is the operator norm closure of the $*$-algebra $\mathbb{C}[Z, A]^G$.
\end{defn}

It is easy to show that the definition of algebraic Roe algebras is independent of the choice of covariant systems (cf. \cite{Willett-Yu-higher-index-book}).

Let us now recall the definition of maximal Roe algebras. To define the maximal norm on the above $*$-algebra $\mathbb{C}[Z, A]^G$, we need some basic concepts of metric spaces. Let $X \subset Z$ be a locally finite subspace of a metric space $X$. The subspace $Y$ is said to be a net of $Z$ if $Y$ is a locally finite subspace $Z$ and $Z=N_r(Y)=\{z\in Z: d(z, Y)\leq r\}$ for some $r$. A locally finite metric space $Y$ is said to have bounded geometry if $\sup_{y\in Y}\# B_R(y)< \infty$, where $B_R(y)$ is the ball of radius $R$ centered at $y$. We say that a metric space $Z$ has bounded geometry if $Z$ has a bounded geometry net. The following result can be proved by the similar arguments in \cite[Lemma 3.4]{Gong-Wang-Yu}.
\begin{lem}[{cf. \cite[Lemma 3.4]{Gong-Wang-Yu}}]
Let  $G$ be a countable discrete group, $A$ a $G$-$C^*$-algebra, and $Z$ a proper metric space with bounded geometry endowed with a proper $G$-action by isometries. Let $(C_0(Z), \varphi, H)$ be an admissible system. Then for each $T \in \mathbb{C}[Z, A]^G$ there exists a constant $C>0$ such that
$$\|\pi(T)\|\leq C$$
for any $*$-representation $\pi:\mathbb{C}[Z, A]^G \to B(H')$.
\end{lem}

It follows from the above result that the maximal norm on the $*$-algebra $\mathbb{C}[Z, A]^G$ is well-defined.
\begin{defn}
The maximal Roe algebra, denoted by $C^*_{max}(Z, A)^G$, is defined to be the completion of $\mathbb{C}[Z, A]^G$ under the maximal norm
$$\|T\|_{max}=\sup \left\{\|\pi(T)\|:\pi:\mathbb{C}[Z, A]^G \to B(H') \mbox{~is~a~} *\mbox{-representation}\right\}.$$
\end{defn}

Next, we shall recall the concept of localization algebras.
%We remark that the $K$-theory of localization algebras is  computable.
\begin{defn}\leavevmode
\begin{enumerate}
    \item[(1)] The algebraic maximal localization algebra $\mathbb{C}_{max,L}[Z, A]^G$ is defined to be the $*$-algebra of all uniformly bounded and uniformly continuous functions $f:[0, \infty)\to C^*_{max}(Z, A)^G$ such that
     $$\mbox{propagation}(f(t))\to 0, \text{~~as~~} t \to \infty.$$
    The maximal localization algebra $C_{max,L}^*(Z, A)^G$ is defined to be the completion of $\mathbb{C}_{max,L}[Z, A]^G$ under the norm
    $$\|f\|=\sup_{t\in [0,\infty)} \|f(t)\|_{max},$$
    for all $f \in \mathbb{C}_{max,L}[Z, A]^G$.
   \item[(3)] The algebraic localization algebra $\mathbb{C}_{L}[Z, A]^G$ is defined to be the uniformly bounded and uniformly continuous functions $f:[0, \infty)\to C^*(Z, A)^G$ such that
    $$\mbox{propagation}(f(t))\to 0, \text{~~as~~} t \to \infty.$$
   The localization algebra $C_L^*(Z, A)^G$ is defined to be the completion of $\mathbb{C}_{L}[Z, A]^G$ under the norm
    $$\|f\|=\sup_{t\in [0,\infty)} \|f(t)\|,$$
    for all $f \in \mathbb{C}_{L}[Z, A]^G$.
\end{enumerate}
\end{defn}

%================================================================
\begin{comment}
By the universality of the maximal norm, the identity map on the algebraic Roe algebra $\mathbb{C}[Z,A]^G$ extends to a $*$-homomorphism
$$\lambda: C_{max}^*(Z, A)^G \to C^*(Z, A)^G.$$
Similarly, we have a $*$-homomorphism
$$\lambda_L: C_{max, L}^*(Z, A)^G \to C_L^*(Z, A)^G$$
extended from the identity map on the algebraic localization algebras.
We would like to point out that it is usually hard to determine when the $*$-homomorphism $\lambda$ is isomorphic, while the map $\lambda_{L}$ always induces an isomorphism on $K$-theory. The following result is a consequence of the Mayer--Vietoris sequence and the the five lemma arguments.
\begin{prop}
Let $Z$ be any finite-dimensional simplicial complex and $G$ a countable discrete group. Assume that $X$ admits a proper $G$-action by isometries, then the map
$$
\lambda_{L,*}: K_*(C_{max, L}^*(Z, A)^G) \to K_*(C_L^*(Z, A)^G)
$$
induced by $\lambda_L: C_{max, L}^*(Z, A)^G \to C_L^*(Z, A)^G$ on $K$-theory is an isomorphism.
\end{prop}
\end{comment}
%================================================
Naturally, we have the evaluation map from the maximal localization algebra to the maximal Roe algebra
$$
e: C_{max,L}^*(Z, A)^G \to C_{max}^*(Z, A)^G
$$
by
$$e(f)=f(0)$$
for all $f \in C_{max,L}^*(Z, A)^G$.
Similarly, we have the evaluation map
$$
e:C_{L}^*(Z, A)^G \to C^*(Z, A)^G.
$$
These evaluation maps induce homomorphisms
$$
e_{*}:K_*(C_{max,L}^*(Z, A)^G) \to K_*(C_{max}^*(Z, A)^G)
$$
and
$$
e_{*}:K_*(C_{L}^*(Z, A)^G) \to K_*(C^*(Z, A)^G).
$$
at the level of $K$-theory.

\subsection{Rips complex}
In this subsection, we review the definition of Rips complexes of a countable discrete group $\Gamma$, and the construction of a model of the universal space for proper $\Gamma$-actions by using Rips complexes.

Let $\Gamma$ be any countable discrete group. A proper $\Gamma$-space $\underline{E}\Gamma$ is said to be universal if it is a metrizable with the quotient space $\underline{E}\Gamma/\Gamma$ paracompact and if for every proper metrizable $\Gamma$-space $X$ with $X/\Gamma$ paracompact then there is a $\Gamma$-equivariant continuous map $X\to \underline{E}\Gamma$, unique up to $\Gamma$-equivariant homotopy.

 % In this section, we shall choose the union of Rips complex $\cup_{s>0}P_s(G)$ as a model of the universal $G$-space $\underline{E}G$.

Let us recall the definition of Rips complexes. For brevity, we assume that the groups we consider are finitely generated.
\begin{defn}  Let $\Gamma$ be a finitely generated group with a word length metric $d$. Let $s\geqslant 0$. The \textit{Rips complex} of $\Gamma$ at scale $s$, denoted $P_s(\Gamma)$, is the simplicial complex with vertex set $\Gamma$, and a subset $\{\gamma_0,\cdots,\gamma_n\}$ of $\Gamma$ spans a simplex if and only if $d(\gamma_i,\gamma_j)\leqslant s$ for all $i,j$.
\end{defn}

Each Rips complex $P_s(\Gamma)$ is equipped with the spherical metric. Recall that the spherical metric is the maximal metric whose restriction to each simplex $\left\{\sum_{i=0}^n c_i t_i\right\} \subset P_s(\Gamma)$ is the metric obtained by identifying this simplex with
$$S^n_+=\left\{(t_0,t_1,\cdots,t_n): \sum_{i=0}^n t^2_i=1, t_i\geq 0, ~\forall~ 0\leq i \leq n\right\}\subset \mathbb{R}^{n+1}$$
by
$$
\left(c_0,c_1,\cdots, c_n\right) \mapsto \left(\frac{c_0}{\sqrt{\sum_{i=0}^n c_i^2} }, \frac{c_1}{\sqrt{\sum_{i=0}^n c_i^2} },\cdots, \frac{c_n}{\sqrt{\sum_{i=0}^n c_i^2} }\right)
$$
where $S^n_+\subset \mathbb{R}^{n+1}$ is equipped with the standard Riemannian metric.

Now we define a $\Gamma$-action on $P_s(\Gamma)$. For each $x=\sum_{\gamma\in\Gamma}t_\gamma\gamma\in P_s(\Gamma)$ and $g \in \Gamma$, define
$$g \cdot (\sum_{\gamma\in\Gamma}t_\gamma\gamma)=\sum_{\gamma\in\Gamma}t_\gamma g\gamma.$$
 It is obvious that this $\Gamma$-action is proper.

Let $G$ and $\Gamma$ be finitely generated groups, and $h:G \to \Gamma$ a group homomorphism. Assume that $S \subset G$ is a finite and symmetric generating subset of $G$ in the sense that $S$ is finite and $g^{-1} \in S$ for each $g \in S$. One can define a left invariant word length metric $d_G$ on $G$ associated to the generating subset $S$. In addition, there exists a finite and symmetric generating subset $S' \subset \Gamma$ of $\Gamma$ containing $h(S)$. One obtains a left invariant metric $d_{\Gamma}$ on $\Gamma$ such that $d_{\Gamma}(h(g_1),h(g_2))\leq d_G(g_1,g_2)$ for all $g_1,g_2$ in $G$.
 For each $s>0$, the map $h$ extends to a continuous map
$$h: P_s(G) \to P_s(\Gamma),$$
by
$$h(\sum_{\gamma\in\Gamma}t_\gamma\gamma)=\sum_{\gamma\in\Gamma}t_\gamma h(\gamma)$$
for each $\sum_{\gamma\in\Gamma}t_\gamma\gamma \in P_s(\Gamma)$. Note that
$$d_{P_s(\Gamma)}(h(x),h(y))\leq d_{P_s(G)}(x, y)$$
for all $s>0$ and all $x, y \in P_s(G)$.

\subsection{Relative Roe algebras and relative localization algebras}\label{section-relative-Roe}
In this section, we  will define the relative Roe algebra and relative localization algebra associated with a group homomorphism $h: G \to \Gamma$.
\begin{defn}
A $C^*$-algebra $A$ is called a $(G,\Gamma)$-$C^*$-algebra if $A$ is a $G$-$C^*$-algebra and $\Gamma$-$C^*$-algebra simultaneously, and $g\cdot a=h(g)\cdot a$ for all $g\in G$, $a\in A$.
\end{defn}

Note that if $A$ is a $(G,\Gamma)$-$C^\ast$-algebra, the restriction of the $G$-action to the subgroup $\ker(h)\subset G$ is trivial.

There are natural $*$-homomorphisms (cf. \cite{Willett-Yu-higher-index-book})
$$
h_{max, s}: C^*_{max}(P_s(G), A)^G \to C^*_{max}(P_s(\Gamma), A)^{\Gamma}
$$
and
$$
h_{max, L, s}: C^*_{max, L}(P_s(G), A)^G \to C^*_{max, L}(P_s(\Gamma), A)^{\Gamma},
$$

If the homomorphism $h: G \to \Gamma$ has amenable kernel, then $h$ induces a homomorphism from $C_{red}^\ast(G)$ to $C^\ast_{red}(\Gamma)$.  Indeed, since the trivial representation of $\ker(h)$ is weakly contained in the regular representation of $\ker(h)$, by continuity of induction the regular representation of $G/\ker(h)$ is weakly contained in the regular representation of $G$. It follows that the reduced $C^*$-algebra of $G$ maps onto the reduced $C^\ast$-algebra of $h(G)\cong G/\ker(h)$. Moreover, the inclusion of $h(G)$ into $\Gamma$ induces an injective $*$-homomorphism from $C^*_{red}(h(G))$ into $C^*_{red}(\Gamma)$. By identifying equivariant Roe algebras with the stabilization of reduced group $C^*$-algebras (see Section 5.3 in \cite{Willett-Yu-higher-index-book}),  we also have $*$-homomorphisms
$$
h_{red, s}: C^*(P_s(G), A)^G \to C^*(P_s(\Gamma), A)^{\Gamma}
$$
and
$$
h_{red, L, s}: C^*_{L}(P_s(G), A)^G \to C^*_{L}(P_s(\Gamma), A)^{\Gamma}
$$
in the reduced cases.

To define relative Roe algebras, we shall recall the concept of the mapping cone associated to a $*$-homomorphism between $C^*$-algebras.
\begin{defn}
Given a $*$-homomorphism $h:A\rightarrow B$ between $C^*$-algebras $A$ and $B$, the  \textit{mapping cone} $C_h$ of $h$ is defined to be $$C_h:=\{(a,f)| a\in A, f\in C_0([0,1),B), h(a)=f(0)\}.$$
\end{defn}

Let $i: C_0(0,1) \otimes B \to C_h$ be the $*$-homomorphism defined by $i(f)=(0,f)$ for all $a \in C_0((0,1), B)$, and $j: C_h\to A$ the $*$-homomorphism by $j(a,f)=a$ for all $(a,f) \in C_h$, where $C_0(0,1) \otimes B$ is the suspension of $B$ which can be viewed as the $C^*$-algebra of all continuous functions from the open interval $(0,1)$ to $B$ that vanish at two ends. The short exact sequence
$$
0 \to C_0(0,1)\otimes B \xlongrightarrow{i} C_h \xlongrightarrow{j} A \to 0
$$
induces the following six-term long exact sequence:
\begin{equation*}
    \xymatrix{
    K_{1}(B)\ar[r] & K_0(C_h) \ar[r] & K_0(A) \ar[d] \\
    K_{1}(A)\ar[u] & K_{1}(C_h) \ar[l] & K_{0}(B)\ar[l] }
\end{equation*}

With the $*$-homomorphisms between Roe algebras and localization algebras, we can define the relative Roe algebras and the relative localization algebras as the suspension of the mapping cones of those $*$-homomorphisms.
\begin{defn}
Let $h: G \to \Gamma$ be a homomorphism between finitely generated groups and $A$ a $G$-$\Gamma$-$C^*$-algebra. Let $s>0$.
\begin{enumerate}
    \item[(1)] The maximal relative Roe algebra $C_{max}^*(P_s(G), P_s(\Gamma), A)^{G,\Gamma}$ is defined to be the suspension of the mapping cone associated with the $*$-homomorphism
   $$h_{max, s}: C_{max}^*(P_s(G), A)^{G} \to C_{max}^*(P_s(\Gamma), A)^{\Gamma}.$$
    \item[(2)] The maximal relative localization algebra $C_{max,L}^*(P_s(G), P_s(\Gamma), A)^{G,\Gamma}$ is the suspension of the mapping cone associated with the $*$-homomorphism
    $$h_{max, s,L}: C_{max,L}^*(P_s(G), A)^G \to C_{max,L}^*(P_s(\Gamma), A)^{\Gamma}.$$
\end{enumerate}
\end{defn}
We have the following six-term exact sequences:
\[
\adjustbox{scale=0.95,center}{
  \begin{tikzcd}[column sep=1.2em]
    K_{1}(C_{max}^*(P_s(\Gamma),A)^{\Gamma})\ar[r] & K_1(C_{max}^{*}(P_{s}(G),P_{s}(\Gamma),A)^{G,\Gamma}) \ar[r] & K_0(C_{max}^*(P_s(G),A)^{G}) \ar[d] \\
    K_{1}(C_{max}^*(P_s(G),A)^{G})\ar[u] & K_{0}(C_{max}^{*}(P_{s}(G),P_{s}(\Gamma),A)^{G,\Gamma}) \ar[l] & K_{0}(C_{max}^*(P_s(\Gamma),A)^\Gamma)\ar[l],
 \end{tikzcd}
}
\]
and
\begin{equation*}
	\adjustbox{scale=0.93,center}{
\begin{tikzcd}[column sep=1.2em]
    K_{1}(C_{max,L}^*(P_s(\Gamma), A)^\Gamma)\ar[r] & K_1(C_{max,L}^{*}(P_{s}(G), P_{s}(\Gamma),A)^{G,\Gamma}) \ar[r] & K_0(C_{max,L}^*(P_s(G),A)^{G}) \ar[d] \\
    K_{1}(C_{max,L}^*(P_s(G), A)^{G})\ar[u] & K_{0}(C_{max,L}^{*}(P_{s}(G), P_{s}(\Gamma),A)^{G,\Gamma}) \ar[l] & K_{0}(C_{max,L}^*(P_s(\Gamma),A)^\Gamma)\ar[l].
\end{tikzcd}
}
\end{equation*}
Note that for any $r\leqslant s$, there exist natural inclusions
$$C_{max}^{*}(P_{r}(G), P_{r}(\Gamma), A)^{G,\Gamma}\rightarrow C_{max}^{*}(P_{s}(G), P_{s}(\Gamma), A)^{G,\Gamma},$$
and
$$C_{max,L}^{*}(P_{r}(G), P_{r}(\Gamma), A)^{G,\Gamma}\rightarrow C_{max,L}^{*}(P_{s}(G), P_{s}(\Gamma), A)^{G,\Gamma}.$$
Thus we have an inductive system $\big\{C_{max,L}^{*}(P_{s}(G), P_{s}(\Gamma), A)^{G,\Gamma}\big\}_{s\in[0,\infty)}$. We shall define the relative $K$-homology for the pair $(G, \Gamma)$ using this inductive system.

\begin{defn}
Given a group homomorphism $h\colon G\rightarrow\Gamma$ and a $(G, \Gamma)$-$C^*$-algebra $A$, the relative equivariant $K$-homology with coefficients in $A$ of $h\colon G\to \Gamma$ is defined to be the inductive limit
$$K_{*}^{G,\Gamma}(\underline{E}G,\underline{E}\Gamma,A):=\lim\limits_{r\rightarrow \infty}K_*(C_{max,L}^{*}(P_{r}(G), P_{r}(\Gamma),A)^{G,\Gamma}).$$
\end{defn}

Now we are ready to define the evaluation map from the relative localization algebras to the relative Roe algebras.

For each $(a,f)\in C_{max,L}^{*}(P_{s}(G), P_{s}(\Gamma), A)^{G,\Gamma}$, we can view $a$ as a continuous path $\left(a(t)\right)_{t \in [0,\infty)}$ in $\mathbb{C}_{max, L}[P_s(G), A]^G$ and view $f$ as a collection of continuous paths $\left(f_r(t)\right)_{r \in [0,1], t \in [0,\infty)}$ in $\mathbb{C}_{max,L}[P_s(\Gamma), A]^{\Gamma}$.
By the definition of the norm on the mapping cone, for each $s$ we have a natural evaluation map
$$e:C_{max,L}^{*}(P_{s}(G), P_{s}(\Gamma),A)^{G,\Gamma} \to C_{max}^{*}(P_{s}(G), P_{s}(\Gamma),A)^{G,\Gamma}$$
defined by
$$e(a,f)=(a(0), f(0)),$$
for all $(a, f) \in C_{max,L}^{*}(P_{r}(G), P_{r}(\Gamma),A)^{G, \Gamma}$.
Passing to inductive limit, we have a homomorphism
$$e_*: K^{G,\Gamma}_*(\underline{E}G,\underline{E}\Gamma,A) \to \lim\limits_{s \to \infty}K_*(C^*_{max}(P_s(G), P_s(\Gamma), A)^{G,\Gamma}).$$

Similarly, we can define the reduced relative Roe algebras and relative localization algebras in the case when the kernel of the homomorphism $h: G \to \Gamma$ is amenable.
\begin{defn}
Let $h:G\rightarrow\Gamma$ be a homomorphism between finitely generated groups with $\ker(h)$ amenable and $A$ a $(G,\Gamma)$-$C^*$-algebra. Let $s>0$.
\begin{enumerate}
  \item[(1)] The relative Roe algebra, $C^{*}(P_{s}(G), P_{s}(\Gamma),A)^{G,\Gamma}$, is defined to be the suspension of the mapping cone of the $*$-homomorphism
$$h_{red, s}:C^*(P_s(G), A)^G \to C^*(P_s(\Gamma), A)^{\Gamma}.$$
  \item[(2)] The relative localization algebra, $C^{*}_L(P_{s}(G), P_{s}(\Gamma),A)^{G, \Gamma}$, to be the suspension of the mapping cone of the $*$-homomorphism
$$h_{red, s,L}:C_L^*(P_s(G), A)^G \to C_L^*(P_s(\Gamma), A)^{\Gamma}.$$
\end{enumerate}

\end{defn}

We have defined the relative $K$-homology groups using maximal localization algebras. We can also consider the reduced relative localization algebras when the group homomorphism $h: G \to \Gamma$ has amenable kernel. In fact, these two relative $K$-homology groups coincide. Following the same arguments in the proof of \cite[Theorem 3.2]{Yu_localization}, the $K$-theory of the maximal localization algebra is identical with the $K$-theory of the reduced localization algebra. It follows from the five lemma that
 $$K_*(C_{L}^{*}(P_{s}(G), P_{s}(\Gamma),A)^{G,\Gamma})\cong K_*(C_{max,L}^{*}(P_{s}(G), P_{s}(\Gamma),A)^{G,\Gamma}).$$
As a result, we have that
$$K_{*}^{G,\Gamma}(\underline{E}G,\underline{E}\Gamma,A)=\lim\limits_{s\rightarrow \infty}K_*(C_{L}^{*}(P_{s}(G), P_{s}(\Gamma),A)^{G,\Gamma}),$$
when $h: G \to \Gamma$ has  amenable  kernel.
Moreover, in this case, we also have the evaluation map
$$e:C_{L}^{*}(P_{s}(G), P_{r}(\Gamma),A)^{G,\Gamma} \to C^{*}(P_{s}(G), P_{s}(\Gamma),A)^{G,\Gamma}$$
defined by
$$e(a,f)=(a(0), f(0)),$$
for all $(a, f) \in C_{L}^{*}(P_{s}(G), P_{s}(\Gamma),A)^{G, \Gamma}$, which induces the homomorphism
$$e_*: K^{G,\Gamma}_*(\underline{E}G,\underline{E}\Gamma,A) \to \lim\limits_{s \to \infty}K_*(C^*(P_s(G), P_s(\Gamma), A)^{G,\Gamma}),$$

%\subsection{Relative group $C^*$-algebras}
%
%In this subsection, we shall establish the relationship between relative Roe algebras and relative group $C^*$-algebras associated to a group homomorphism.
%
%
%Let $A$ be a $C^*$-algebra with an action of a countable discrete group $G$. Denote by $C_c(G, A)$ the set of all $A$-valued functions on $G$ with compact support. Each element in $C_c(G, A)$ can be expressed as a finite sum $\sum_{g \in G} a_g g$. In fact, $C_c(G, A)$ is a $*$-algebra endowed with the multiplication
%$$\left(\sum_{g \in G} a_g g\right)\left(\sum_{g \in G} b_g g\right)=\sum_{st}a_s s(b_t)st$$
%and the $*$-operation
%$$\left(\sum_{g \in G} a_g g\right)^*=\sum_{g \in G} g^{-1}(a^*_g)g^{-1},$$
%for all $\sum_{g \in G} a_g g, \sum_{g \in G} b_g g \in C_c(G, A)$.
%Denote by $C^*_{max}(G, A)$ and $C^*_{red}(G, A)$ the $C^*$-algebras under the maximal and the reduced $*$-norm on $C_c(G, A)$, respectively.

%Let $G$ and $\Gamma$ be finitely generated groups and $h:G\rightarrow\Gamma$ a group homomorphism. For any $(G,\Gamma)$-$C^*$-algebra $A$,
Let $C^*_{max}(G, A)$ and $C^*_{red}(G, A)$ be the maximal and reduced $C^*$-crossed product of $G$ and $A$, respectively. Then
$h\colon G\to \Gamma$ induces a $*$-homomorphism
$$h_{max}:C_{max}^*(G,A)\rightarrow C_{max}^*(\Gamma,A).$$
If $h$ has amenable  kernel, then $h$ induces a $*$-homomorphism $$h:C_{red}^*(G,A)\rightarrow C_{red}^*(\Gamma,A).$$
We define $C_{max}^*(G,\Gamma,A)$ to be the suspension of the mapping cone of $h_{max}$, and   we call it the $\emph{maximal relative group C*-algebra}$ of $(G,\Gamma)$ with coefficients in $A$.
If the  kernel of $h$ is amenable,  then we can likewise define the \emph{reduced relative group $C^*$-algebra} $C_{red}^*(G,\Gamma, A)$. We have
$$
C^*_{max}(P_s(G), P_s(\Gamma), A)^{G, \Gamma} \cong C^*_{max}(G, \Gamma, A)\otimes K
$$
for each $s>0$, where $K$ is the algebra of compact operators on Hilbert space.
If the homomorphism $h: G \to \Gamma$ has amenable kernel, we have that
$$
C^*_{red}(P_s(G), P_s(\Gamma), A)^{G, \Gamma} \cong C^*_{red}(G, \Gamma, A)\otimes K,
$$
for each $s>0$.

\begin{comment}
We can also get the reduced analogues of Proposition \ref{prop-relative-Roe-and-relarive-group-C-alg} and Corollary \ref{cro-commute-relative-Roe-and-group-alg}. For all $r<s$, we have the commutative diagram
\begin{equation*}
  \begin{tikzcd}
   K_{*}(C^{*}(P_{r}(G), P_{r}(\Gamma),A)^{G,\Gamma}) \ar{d} \ar{r}&K_{*}(C_{red}^{*}(G,\Gamma,A))\ar[equal]{d}\\
 K_{*}(C^{*}(P_{s}(G), P_{s}(\Gamma),A)^{G,\Gamma}) \ar{r}& K_{*}(C_{red}^{*}(G,\Gamma,A))\\
 \end{tikzcd}
\end{equation*}
when $h:G\rightarrow \Gamma$ has amenable  kernel.
\end{comment}

\begin{defn}\leavevmode
\begin{enumerate}
    \item[(1)] The \emph{maximal relative Baum-Connes assembly map} $$\mu_{max}:K_{*}^{G,\Gamma}(\underline{E}G,\underline{E}\Gamma,A)\rightarrow K_*(C_{max}^*(G,\Gamma,A))$$ is defined to be the homomorphism
        $$e_*:\lim\limits_{s\to \infty}K_*(C^*_{max,L}(P_s(G), P_s(\Gamma), A)^{G,\Gamma})\to \lim\limits_{s\to \infty}K_*(C_{max}^*(P_s(G), P_s(\Gamma), A)^{G,\Gamma}).$$
    \item[(2)] When the homomorphism $h:G\rightarrow \Gamma$ has amenable  kernel, the \emph{reduced relative Baum-Connes assembly map}  $$\mu_{red}:K_{*}^{G,\Gamma}(\underline{E}G,\underline{E}\Gamma,A)\rightarrow K_*(C_{red}^*(G,\Gamma,A)),$$
   is defined to be the homomorphism
    $$e_*:\lim\limits_{s\to \infty}K_*(C^*_L(P_s(G), P_s(\Gamma), A)^{G,\Gamma})\to \lim\limits_{s\to \infty}K_*(C^*(P_s(G), P_s(\Gamma), A)^{G,\Gamma}).$$
\end{enumerate}
\end{defn}

Let us state  the maximal and reduced strong relative Novikov conjectures.

\begin{conj}[Maximal Strong Relative Novikov Conjecture]\label{conj:max}
Let $h:G\rightarrow \Gamma$ be a group homomorphism between countable discrete groups $G$ and $\Gamma$, and $A$ a $(G, \Gamma)$-$C^*$-algebra. The \emph{maximal relative Baum-Connes assembly map} $$\mu_{max}:K_{*}^{G,\Gamma}(\underline{E}G,\underline{E}\Gamma, A)\rightarrow K_*(C_{max}^*(G,\Gamma, A))$$
is injective.
\end{conj}
We can also define a reduced analogue of the above conjecture. Since the reduced relative group $C^*$-algebra is not defined for general homomorphisms between groups, we shall state the reduced version of the conjecture under the extra assumption that the group homomorphism $h: G \to \Gamma$ has amenable kernel.

\begin{conj}[Reduced Strong Relative Novikov Conjecture]\label{conj:reduce}
Let $h:G\rightarrow \Gamma$ be a group homomorphism between countable discrete groups $G$ and $\Gamma$ and $A$ a $(G, \Gamma)$-$C^*$-algebra. When the homomorphism $h:G\rightarrow \Gamma$ has amenable  kernel, the \emph{reduced relative Baum-Connes assembly map} $$\mu_{red}:K_{*}^{G,\Gamma}(\underline{E}G,\underline{E}\Gamma)\rightarrow K_*(C_{red}^*(G,\Gamma))$$ is injective.\\
\end{conj}
In Section 3, we shall define a reduced strong relative Novikov conjecture for all group homomorphisms with the help of ${\rm II}_1$-factors.

When the group $G$ is trivial, it follows from the five lemma that the injectivity of the relative assembly map is equivalent to  the injectivity of the following (absolute) assembly map
$$\mu_{red}: K^{\Gamma}_*(\underline{E}\Gamma) \to K_*(C^*_{red}(\Gamma)).$$
The injectivity of $\mu_{red}$ is the usual strong Novikov conjecture, which has been verified for a large class of groups  \cite{Apa-Valett-Julg-BC-survey,Chen-Fu-Wang-zhou, CM,Chen-Wang-Wang-FCE, KasSkan, Laff-hyperbolic, Skan_Tu_Yu, Yu_fin_asym_dim, Yu_coa_emb, Yu_survey_on_Novikov}.

\section{Relative Baum--Connes conjecture}
In this section, we shall introduce a reduced relative Baum--Connes assembly map for all group homomorphisms $h\colon G\to \Gamma$ between countable discrete groups with the help of ${\rm II}_1$-factors. We show that this  relative Baum--Connes assembly map is an isomorphism when both $G$ and $\Gamma$ are  hyperbolic groups.
\subsection{Relative reduced assembly map}\label{sect-red-assembly-map}

It is known that for any countable discrete group $G$, there exists a ${\rm II}_1$-factor $\mathcal M$ such that there is a trace-preserving embedding $\phi: C^*_{red}(G) \hookrightarrow \mathcal M$. Indeed, every group $G$ can be viewed as a subgroup of an ${\rm ICC}$ group\footnote{A group is said to be an ICC group, or to have the infinite conjugacy class property, if the conjugacy class of every  element but the identity element is infinite.} $G'$ whose group von Neumann algebra $L(G')$ is a $\rm II_1$-factor. We can take $\mathcal  M=L(G')$.

Now we shall use the embedding $\phi$ to define the reduced relative assembly map. Let us equip $\mathcal M$ with the trivial actions of $G$ and $\Gamma$. For any group homomorphism $h:G\rightarrow\Gamma$,  we  define a $*$-homomorphism from  $C_{red}^*(G, \mathcal M)$ to $C_{red}^*(\Gamma,\mathcal M\otimes \mathcal M)$ as follows.

\begin{lem}\label{lemmakey}
 Let $h:G\rightarrow\Gamma$ be any group homomorphism between countable discrete groups and  $\mathcal M$ a  ${\rm II}_1$-factor endowed with  a trace-preserving embedding  $\phi: C^*_{red}(G)\to \mathcal M$.
Then there exists a $*$-homomorphism
$$h_{red,\mathcal M}: C_{red}^*(G, \mathcal M) \rightarrow  C_{red}^*(\Gamma,\mathcal M\otimes \mathcal M)$$ by
  $$h_{red, \mathcal M}(\sum a_gg)=\sum \left(a_g\otimes \phi(g) \right)h(g) $$
  for all $\sum a_gg \in C^*_{red}(G, \mathcal M)$, where  the actions of $G$ and $\Gamma$ on $\mathcal M$ are trivial.
\end{lem}
\begin{proof} Given a group homomorphism $h:G\rightarrow\Gamma$, we consider the map
$$G\stackrel{h^1}{\longrightarrow} G\times\Gamma$$
defined by
$$h^1(g)=(g,h(g)),$$
for all $g \in G$. Notice that $h^1$ is injective, thus $h^1$ induces a $*$-homomorphism
$$
h^1_{red}: C_{red}^*(G, \mathcal M) \to C^*_{red}(G \times \Gamma, \mathcal M).
$$
Therefore, we have that
 \begin{eqnarray*}
C_{red}^*(G, \mathcal M)&\stackrel{{h^1_{red}}}{\longrightarrow}&C_{red}^*(G\times\Gamma, \mathcal M)\\
&\stackrel{\cong}{\longrightarrow} &C_{red}^*(G)\otimes C_{red}^*(\Gamma, \mathcal M)\\
&\stackrel{\phi\otimes id}{\longrightarrow}&\mathcal M \otimes C_{red}^*(\Gamma, \mathcal M)\\
&\stackrel{\cong}{\longrightarrow}& C_{red}^*(\Gamma,\mathcal M\otimes \mathcal M).\\
\end{eqnarray*}
The composition of the above maps gives  the map $h_{red, \mathcal M}$.
\end{proof}

We denote by $C_{h_{red,\mathcal M}}$  the mapping cone of the map $h_{red, \mathcal  M}:  C_{red}^*(G, \mathcal M) \rightarrow  C_{red}^*(\Gamma,\mathcal M\otimes \mathcal M)$.
\begin{defn}
Define the relative group $C^*$-algebra with coefficients in the ${\rm II}_1$-factor $\mathcal{M}$ to be
  $$
  C_{red}^*(G, \Gamma, \mathcal M)=C_0(\mathbb{R})\otimes   C_{h_{red, \mathcal M}}
  .$$
\end{defn}

\begin{rmk}
\begin{enumerate}
  \item  It is known that $K_0(\mathcal{M})=\mathbb{R}$ and $K_1(\mathcal{M})=0$ for any ${\rm II}_1$-factor $\mathcal{M}$. If the group $G$ is amenable, then the group $C^*$-algebra $C^*_{red}(G)$ is nuclear. By K\"{u}nneth's formula for $K$-theory of operator algebras, we have that
  $$
  K_*(C^*_{red}(G, \mathcal{M}))\cong K_*(C^*_{red}(G))\otimes\mathbb{R}.
  $$
 %In fact, the isomorphism between $K$-theory is induced by the homomorphism $\mbox{id}\otimes I: A\to A \otimes \mathcal{M}$, where $I \in\mathcal{M}$ is the identity element and $\mbox{id}: A \to A$ is the identity map.  The inverse of the induced map on $K$-theory is the map induced by $\mbox{id}\otimes \mbox{Tr}: A\otimes \mathcal{M}\to A$, where $\mbox{Tr}: \mathcal{M}\to \mathbb{C}$ is the canonical trace of $\mathcal{M}$.
  \item We remark here that the assumption that the embedding $\phi: C_{red}(G)\to \mathcal{M}$ is trace-preserving is crucial to encode the data of the homomorphism $h:G \to \Gamma$ in the definition of relative group $C^*$-algebras. In the case when $G$ and $\Gamma$ are amenable, the homomorphism
      $$h_{red, \mathcal  M, *}:  K_*(C_{red}^*(G, \mathcal M)) \rightarrow  K_*(C_{red}^*(\Gamma,\mathcal M\otimes \mathcal M))$$
 induced by the map $h_{red, \mathcal  M}:  C_{red}^*(G, \mathcal M) \rightarrow  C_{red}^*(\Gamma,\mathcal M\otimes \mathcal M)$ is the same as the map
 $$
 h_*\otimes \mbox{id}:K_*(C^*_{red}(G))\otimes \mathbb{R}\to K_*(C_{red}^*(\Gamma))\otimes \mathbb{R},
 $$
 where $h_*:K_*(C^*_{red}(G))\to K_*(C_{red}^*(\Gamma))$ is the map induced by the $*$-homomorphism $h: C^*_{red}(G)\to C^*_{red}(\Gamma)$.

 %On the contrast, if $G$ and $\Gamma$ are amenable and the embedding $\phi: C^*_{red}(G)\to \mathcal{M}$ induces the zero map on $K$-theory, it follows from the six-term exact sequence that
 %$$ K_*(C^*(G, \Gamma, \mathcal M)^{G, \Gamma})=K_{*+1}(C^*_{red}(G), \mathcal{M})\oplus K_{*}(C^*_{red}(\Gamma), \mathcal{M} \otimes \mathcal{M}),$$for $*=0, 1$.  In this case the $K$-theory of the relative group $C^*$-algebra does not encode any information of the homomorphism $h: G \to \Gamma$.
\end{enumerate}
\end{rmk}

Now let us introduce the relative K-homology with coefficients in a ${\rm II}_1$-factor. Following the construction of the $*$-homomorphism between localization algebras, we have a $*$-homomorphism
$$
h_{L, \mathcal M}:C_{L}^*(P_s(G), \mathcal M)^G \to C_{L}^*(P_{s}(\Gamma), \mathcal M \otimes \mathcal M)^{\Gamma}.
$$
Let $C_{h_{L,\mathcal M}}$ be the mapping cone of the $*$-homomorphism $h_{L,\mathcal M}$.
For each $s>0$, we define the relative localization algebra with coefficients in $\mathcal M$ to be
$$
C_{L}^{*}(P_s(G), P_{s}(\Gamma), \mathcal M)^{G, \Gamma}=C_0(\mathbb{R})\otimes C_{h_{L,\mathcal M}}.$$
For each $s>0$, there is a natural evaluation map
$$
e: C^*_{L}(P_s(G), P_s(\Gamma), \mathcal M)^{G, \Gamma} \to C^*(P_s(G), P_s(\Gamma), \mathcal M)^{G, \Gamma}
.$$
For all $r<s$, we have the following commutative diagram
\begin{equation*}
\adjustbox{scale=0.95,center}{
	\begin{tikzcd}[column sep=1.3em]
 K_*(C_{L}^{*}(P_{r}(G),P_{r}(\Gamma), \mathcal M)^{G, \Gamma}) \ar{r}{e_*} \ar{d}
& K_*(C^{*}(P_{r}(G), P_{r}(\Gamma), \mathcal M)^{G, \Gamma}) \ar{r}{\cong} \ar{d}& K_*(C_{red}^*(G,\Gamma, \mathcal M)) \ar[equal]{d}\\
 K_*(C_{L}^{*}(P_{s}(G), P_{s}(\Gamma), \mathcal M)^{G, \Gamma}) \ar{r}{e_*}&
 K_*(C^{*}(P_{s}(G), P_{s}(\Gamma), \mathcal M)^{G, \Gamma}) \ar{r}{\cong}& K_*(C_{red}^*(G,\Gamma, \mathcal M))
\end{tikzcd}
}
\end{equation*}
 Thanks to this compactibility we can define the relative $K$-homology and the relative assembly map with coefficients in $\mathcal  M$ as the following inductive limits.

\begin{defn}
 The \textit{relative equivariant K-homology with coefficients in $\mathcal M$} is defined as the inductive limit
 $$K_{*}^{G,\Gamma}(\underline{E}G,\underline{E}\Gamma, \mathcal M)
 :=\lim_{s\to \infty}K_*(C_{L}^{*}(P_s(G), P_{s}(\Gamma), \mathcal M)^{G, \Gamma}).$$
\end{defn}
%We remark that the $K$-homology $K_{*}^{G,\Gamma}(\underline{E}G,\underline{E}\Gamma, \mathcal M)$ is computable.
\begin{defn}
The \emph{reduced  relative Baum-Connes assembly map} $$\mu_{red}:K_{*}^{G,\Gamma}(\underline{E}G,\underline{E}\Gamma, \mathcal M)\rightarrow K_*(C_{red}^*(G,\Gamma,\mathcal  M))$$
is defined to be the inductive limit of the homomorphisms
$$e_*: K_*(C^*_L(P_s(G), P_s(\Gamma), \mathcal M)) \to K_*(C^*(P_s(G), P_s(\Gamma), \mathcal M)).$$

\end{defn}

\begin{conj}[Reduced Strong Relative Novikov Conjecture]
Assume that $h: G \to \Gamma$ is a homomorphism  between countable discrete groups $G$ and $\Gamma$. Let $\mathcal{M}$ be a ${\rm II}_1$-factor such that there exists a trace-preserving embedding $\phi: C^*_{red}(G) \to \mathcal{M}$.
The \emph{reduced relative Baum-Connes assembly map}
$$\mu_{red}:K_{*}^{G,\Gamma}(\underline{E}G,\underline{E}\Gamma, \mathcal M)\rightarrow K_*(C_{red}^*(G,\Gamma, \mathcal M))$$
is injective.
\end{conj}

\begin{conj}[Relative Baum--Connes Conjecture]
Assume that $h: G \to \Gamma$ is a homomorphism  between countable discrete groups $G$ and $\Gamma$. Let $\mathcal{M}$ be a ${\rm II}_1$-factor such that there exists a trace-preserving embedding $\phi: C^*_{red}(G) \to \mathcal{M}$. The \emph{reduced relative Baum-Connes assembly map}
$$\mu_{red}:K_{*}^{G,\Gamma}(\underline{E}G,\underline{E}\Gamma, \mathcal M)\rightarrow K_*(C_{red}^*(G,\Gamma,\mathcal  M))$$
is an isomorphism.
\end{conj}

% make the following a remark.
\begin{rmk}\leavevmode
\begin{enumerate}
\item[(1)] We remark that the definition of the reduced Baum--Connes assembly map for $h\colon G\to \Gamma$ does not depend on the choice of the trace-preserving embedding $\phi: C^*_{red}(G) \to \mathcal{M}$. Note that every Rips complex $P_d(G)$ can be express as a finite union $P_d(G)=\cup_{i} G\cdot X_i$ where each $X_i$ is a precompact and open subset of $P_d(G)$ which is $F_i$-invariant for some finite subgroup $F_i$ of $G$, and $gX_i\cap X_i=\emptyset$ for all $g\in G-F_i$. We have that

    \begin{align*}
    K_*(C^*_L(G\cdot X_i, \mathcal{M})^G) &\cong K_*(C^*_L(X_i, \mathcal{M})^{F_i})\\
                                                                & \cong K_*(C^*(X_i, \mathcal{M})^{F_i})\\
                                                                &  \cong K_*(C^*(X_i)^{F_i})\otimes \mathbb{R}\\
                                                                 &  \cong K_*(C_L^*(X_i)^{F_i})\otimes \mathbb{R}\\
                                                                & \cong K_*(C^*_L(G\cdot X_i, \mathcal{M})^G)\otimes \mathbb{R}.
     \end{align*}
The first and the last equality follow from the definition of the localization algebras. The second and the fourth equality follow from the Baum--Connes conjecture for finite groups and the third equality follows from the K\"{u}nneth formula for $K$-theory of operator algebras. It follows from the six-term exact sequence for the $K$-theory of localization algebras and the five lemma that
$$
K_{*}^G(\underline{E}G,\mathcal M)\cong K_*^G(\underline{E}G)\otimes \mathbb{R}.
$$
%$$K_{*}^{G,\Gamma}(\underline{E}G,\underline{E}\Gamma, \mathcal  M)\cong K_*^{G,\Gamma}(\underline{E}G, \underline{E}\Gamma)\otimes \mathbb{R}.$$
Therefore, the definition of $K$-homology $K_{*}^{G,\Gamma}(\underline{E}G,\underline{E}\Gamma, \mathcal  M)$ does not depend on the choice of the ${\rm II}_1$-factor $\mathcal{M}$ by the five lemma.

\item[(2)] Let $BG$ and $B\Gamma$ be the classifying space for $G$ and $\Gamma$, respectively. There is a natural map
$$h: BG \to B\Gamma$$
induced by the group homomorphism $h: G \to \Gamma$. The injectivity of the reduced relative Baum-Connes assembly map $$\mu_{red}:K_{*}^{G,\Gamma}(\underline{E}G,\underline{E}\Gamma, \mathcal M)\to K_*(C_{red}^*(G,\Gamma, \mathcal M))$$
implies that the relative assembly map
   $$\mu_{red}:K_*(BG,B\Gamma)\otimes \mathbb{R} \to K_*(C^*_{red}(G, \Gamma,\mathcal  M))$$
is injective, the latter of which we shall now review.

Now, let us define the relative $K$-homology group $K_*(BG,B\Gamma)$. Following the construction in Section \ref{section-relative-Roe}, we can construct $*$-homomorphisms
$$h_{L}: C^*_{ L}(BG) \to C^*_{L}(B\Gamma),$$
and
$$
h_{L}: C^*_{ \mathcal{M}, L}(BG, \mathcal{M}) \to C^*_{L}(B\Gamma, \mathcal{M}\otimes \mathcal{M}),
$$
induced by the continuous map $h: BG \to B\Gamma$.
Define the relative $K$-homology group $K_*(BG, B\Gamma)$ to be the $K$-theory of the suspension of the mapping cone associated to the $*$-homomorphism
$$h_{L}: C^*_{ L}(BG) \to C^*_{L}(B\Gamma).$$
By the five lemma, the relative $K$-homology group $K_*(BG, B\Gamma)\otimes \mathbb{R}$ is equivalent to the $K$-theory of the suspension of the mapping cone associated to the $*$-homomorphism
$$h_{\mathcal{M}, L}: C^*_{L}(BG, \mathcal{M}) \to C^*_{L}(B\Gamma, \mathcal{M}\otimes \mathcal{M}).$$
Following the constructions in \cite{Baum-Conn_chern-character}, we have the relative local index map
$$
\sigma_{G, \Gamma}: K_*(BG, B\Gamma)\otimes \mathbb{R} \to K_{*}^{G,\Gamma}(\underline{E}G,\underline{E}\Gamma, \mathcal M).
$$
By the Connes--Chern character \cite{Baum-Conn_chern-character}, we know that the
$ K_*(BG)\otimes \mathbb{R}$ is a summand of the $K$-homology group $ K_{*}^{G}(\underline{E}G, \mathcal M)$. It follows that the relative local index map
$$
\sigma_{G, \Gamma}: K_*(BG, B\Gamma)\otimes \mathbb{R} \to K_{*}^{G,\Gamma}(\underline{E}G,\underline{E}\Gamma, \mathcal M)
$$
is injective.

In summary, we have the composition
\begin{equation*}
  \begin{tikzcd}
   K_*(BG,B\Gamma)\otimes \mathbb{R}  \ar{r}{\sigma_{G,\Gamma}} &K_{*}^{G,\Gamma}(\underline{E}G,\underline{E}\Gamma, \mathcal M) \ar{r}{\mu_{red}} & K_*(C^*_{red}(G, \Gamma,\mathcal  M))
  \end{tikzcd}
\end{equation*}
which we still call the relative assembly map. For simplicity, we also denote it by $\mu_{red}$. As a result, the injectivity of the reduced relative Baum--Connes assembly map $\mu_{red}:K_{*}^{G,\Gamma}(\underline{E}G,\underline{E}\Gamma, \mathcal M)\to K_*(C_{red}^*(G,\Gamma, \mathcal M))$  implies that the relative assembly map
$$\mu_{red}:K_*(BG,B\Gamma)\otimes \mathbb{R} \to K_*(C^*_{red}(G, \Gamma,\mathcal  M))$$
is injective.

\item[(3)] Similarly, there is also the  maximal relative assembly map
   $$\mu_{max}:K_*(BG,B\Gamma)\otimes \mathbb{\mathbb R} \to K_*(C^*_{max}(G, \Gamma)) \otimes \mathbb{\mathbb R}.$$
  Moreover,  the injectivity of the maximal relative Baum--Connes assembly map $\mu_{max}:K_{*}^{G,\Gamma}(\underline{E}G,\underline{E}\Gamma, \mathcal M)\to K_*(C_{max}^*(G,\Gamma))$ implies the injectivity of $\mu_{max}:K_*(BG,B\Gamma)\otimes \mathbb{R} \to K_*(C^*_{max}(G, \Gamma)) \otimes \mathbb{R}$.
\end{enumerate}
\end{rmk}

\subsection{Relative Baum--Connes conjecture  for hyperbolic groups}
We will conclude this subsection by showing that any pair of hyperbolic groups $(G, \Gamma)$ satisfies the relative Baum--Connes conjecture with coefficients in a ${\rm II}_1$-factor $\mathcal{M}$.

%The notion of hyperbolicity was introduced and developed by Gromov \cite{Gromov-hyperbolic} which was inspired by a geometric property of triangles in the hyperbolic space $\mathbb{H}^n$.

\begin{defn}[Gromov, \cite{Gromov-hyperbolic}]
Let $G$ be a finitely generated group equipped with a left invariant word length metric. The group $G$ is said to be hyperbolic if there exists a constant $\delta>0$ such that each geodesic triangle is $\delta$-thin in the sense that for any $x,y,z \in G$, the geodesic, denoted by $[x,y]$, joining $x$ and $y$, is contained the $\delta$-neighborhood of the union of other two geodesics $[x,z]$ and $[y,z]$.
\end{defn}

Lafforgue showed that the Baum--Connes conjecture with coefficients holds for all  hyperbolic groups \cite{Laff-hyperbolic}.

\begin{thm}[\cite{Laff-hyperbolic}]\label{thm-BC-hyperbolic}
Let $G$ be a hyperbolic group and $A$ any $G$-$C^*$-algebra. Then the Baum--Connes conjecture with coefficients in $A$ holds for $G$, i.e. the Baum--Connes assembly map
$$\mu: K^G_*(\underline{E}G, A) \to K_*(C^*_{red}(G,A))$$
is an isomorphism.
\end{thm}

Combining Lafforgue's theorem  (\cite{Laff-hyperbolic}) with the six-term $K$-theory exact sequence, we show that the relative Baum--Connes conjecture with coefficients in a ${\rm II}_1$ factor holds for a pair of hyperbolic groups.
\begin{prop}\label{prop-relative-BC-for-hyperbolic}
Let $G$ and $\Gamma$ be hyperbolic groups, and $h\colon G \to \Gamma$ a group homomorphism. Let $\phi\colon C^*_{red}(G) \to \mathcal{M}$ be a trace-preserving embedding of $C^*_{red}(G)$ into a ${\rm II}_1$-factor $\mathcal{M}$. Then the  relative Baum--Connes conjecture with coefficients in $\mathcal  M$ holds for $h\colon G\to \Gamma$, i.e., the reduced relative Baum--Connes assembly map
$$\mu_{red}:K_{*}^{G,\Gamma}(\underline{E}G,\underline{E}\Gamma, \mathcal M)\rightarrow K_*(C_{red}^*(G,\Gamma, \mathcal M))$$
is an isomorphism.
\end{prop}

\begin{proof} We have the following commutative diagram:

\begin{equation*}
\begin{tikzcd}
   K_{*+1}^G(\underline{E}G, \mathcal M) \ar{d}\ar{r}{\mu_{ G}}[swap]{\cong}&    K_{*+1}(C^*_{red}(G, \mathcal M)) \ar{d}\\
K_{*+1}^{\Gamma}(\underline{E}\Gamma, \mathcal M \otimes \mathcal M) \ar{ d}\ar{r}{\mu_{\Gamma}}[swap]{\cong}&  K_{*+1}(C^*_{red}(\Gamma, \mathcal M\otimes \mathcal M)) \ar{d}\\
K_{*+1}^{G,\Gamma}(\underline{E}G,\underline{E}\Gamma, \mathcal M)\ar{r}{\mu_{red}}\ar{d}
& K_{*+1}(C^*_{red}(G,\Gamma, \mathcal M))\ar{d}\\
  K_{*}^G(\underline{E}G, \mathcal M) \ar{d}\ar{r}{\mu_{ G}}[swap]{\cong}&    K_{*}(C^*_{red}(G, \mathcal M)) \ar{d}\\
K_{*}^{\Gamma}(\underline{E}\Gamma, \mathcal M\otimes\mathcal  M) \ar{r}{\mu_{\Gamma}}[swap]{\cong}&  K_{*}(C^*_{red}(\Gamma, \mathcal M\otimes \mathcal M)).
\end{tikzcd}
\end{equation*}

By Theorem \ref{thm-BC-hyperbolic}, the assembly maps $\mu_{G}$ and $\mu_{\Gamma}$ are isomorphic. It follows from the five lemma that the relative assembly map
$$
\mu_{red}: K_{*}^{G,\Gamma}(\underline{E}G,\underline{E}\Gamma, \mathcal M)\to
K_{*}(C^*_{red}(G,\Gamma, \mathcal M))
$$
is an isomorphism. This finishes the proof.
\end{proof}

Using the same arguments above, we can generalize Proposition \ref{prop-relative-BC-for-hyperbolic} to the following result.
\begin{prop}\label{thm-relative-real-BC}
  Let $G$ and $\Gamma$ be any discrete groups and $h\colon G \to \Gamma$ a group homomorphism. Let $\phi: C^*_{red}(G) \hookrightarrow \mathcal M$ be a trace preserving embedding of $C^*_{red}(G)$ into a ${\rm II}_1$-factor $\mathcal{M}$. Assume that the Baum--Connes conjecture holds for $G$ and $\Gamma$. Then the relative Baum--Connes conjecture with coefficients in $\mathcal M$  holds for  $h\colon G \to \Gamma$.
\end{prop}

\section{A relative Bott Periodicity}
In this section, we shall prove a Bott periodicity for the relative Roe algebras associated with a pair of groups $(G, \Gamma)$ with $\Gamma$ coarsely embeddable into Hilbert space.

\subsection{$C^*$-algebras associated with Hilbert spaces}\label{section-C-alg-for-Hilbert-space}

Let $E$ be a separable, infinite-dimensional Euclidean space. For any finite-dimensional, affine subspace  $E_a$, denote by $E_a^0$ the finite-dimensional linear subspace of $E$
consisting of differences of elements in $E_a$. Let $\mathcal{C}(E_a)$ be the $\mathbb{Z}_2$-graded
$C^*$-algebra of continuous functions from $E_a$ to the complexified Clifford algebra of $E^0_a$ vanishing at infinity. A $\mathbb{Z}_2$-grading on $\mathcal{C}(E_a)$ is induced from the even and odd parts of $\text{Cliff}(E_a^0)$.

Let $\mathcal{S}$ be the $\mathbb{Z}_2$-graded $C^*$-algebra of all continuous functions on $\mathbb{R}$ vanishing at infinity, where $\mathcal{S}$ is graded according to odd and even functions. Let $\mathcal{A}(E_a)$ be the graded tensor product $\mathcal{S} \widehat{\otimes}\mathcal{C}(E_a)$.

For a pair of finite-dimensional, affine subspaces $E_a$ and $E_b$ with $E_a \subset E_b$, there exists a decomposition
$$E_b=E^0_{ba} \oplus E_a,$$
where $E^0_{ba}$ is the orthogonal complement of $E^0_a$ in $E^0_b$. For each element $v_b \in E_b$, there exists a unique decomposition $v_b=v_{ba}+v_a$, for some $v_{ba} \in E^0_{ba}$ and $v_a \in E_a$.

For each function $h \in \mathcal{C}(E_a)$, we can extend it to a function on $E_b$ via $\tilde{h}(v_b)=h(v_a)$, for all $v_b=v_{ba}+v_a$. The decomposition $E_b=E^0_{ba} \oplus E_a$ gives rise to a Clifford algebra valued function, denoted by $C_{ba}: E_b \to\mbox{ Cliff}(E^0_b)$ on $E_b$ which maps $v_b \in E_b$ to $v_{ba} \in E^0_{ba} \subset \mbox{Cliff}(E_b^0)$.

Denote by $X: \mathcal{S} \to \mathcal{S}$ the operator of multiplication by $ x$ on $\mathbb{R}$.  Note that $X$ is a degree one, essentially selfadjoint, unbounded multiplier of $\mathcal{S}$ with domain the compactly supported functions in $\mathcal{S}$.

 \begin{defn}[\cite{KaHiTr}]\leavevmode
\begin{enumerate}
\item Let $E_a \subset E_b$ be a pair of finite-dimensional, affine subspaces of $E$. One can define a homomorphism
$$
\beta_{ba}: \mathcal{A}(E_a) \to \mathcal{A}(E_b)
$$
by $$\beta_{ba}(f \widehat{\otimes}h)=f(X\widehat{\otimes }1+ 1\widehat{\otimes}C_{ba}) (1 \widehat{\otimes} \tilde{h})$$
for all $f \in \mathcal{S}$, $h \in \mathcal{C}(E_a)$.
\item We define a $C^*$-algebra
$$\mathcal{A}(E):=\varinjlim \mathcal{A}(E_a),$$
where the direct limit is  over all finite-dimensional affine subspaces.
 \end{enumerate}
\end{defn}

%If $E_a \subset E_b \subset E_c$, then we have $\beta_{cb}\circ \beta_{ba}=\beta_{ca} $, therefore the direct limit is well-defined.

\begin{comment}
Now let us define the Bott map by examples. For any $f \in C_0(\mathbb{R})$, we can define an element
$$f(s,v) \in \mathcal{S}\widehat{\otimes} C_0(V, \mbox{Cliff}(H))$$
by
$$f(s,v)=f(s \widehat{\otimes}1 + 1 \widehat{\otimes} v),$$
where $s$ and $v$ can be viewed as the unbounded multiplication operator on $\mathcal{S}$ and $C_0(H, \text{Cliff}(H))$, respectively.

More concretely, for any $f \in \mathcal{S}$, $f(s,v)$ can be defined as follows,
\begin{enumerate}
\item if $f(t)=g(t^2)$, for some $g \in \mathcal{S}$, $f(s,v)=g(s^2 +\|v\|^2)$;
\item if $f(t)=tg(t^2)$, for some $g \in \mathcal{S}$, $f(s,v)=g(s^2 +\|v\|^2)(s \widehat{\otimes}1 + 1 \widehat{\otimes} v)$.
\end{enumerate}
\end{comment}
Given any discrete group $\Gamma$, $\mathcal{S}$ is equipped with trivial $\Gamma$-action. If $\Gamma$ acts on the Euclidean space $E$ by linear isometries, then the $\Gamma$-action on $E$ induces a $\Gamma$-action on the $C^*$-algebra $\mathcal{A}(E)$.
Note that $\mathcal{A}(\{0\}) = \mathcal{S}$. For each $f \in \mathcal{S}$, let $\beta_t(f)=f_t(X\widehat{\otimes}1 + 1 \widehat{\otimes}C)$ for every $t \in [1, \infty)$, where $f_t(x)=f(x/t)$.

We define the Bott map
$$\beta_*:K_*(C^*_{max}(\Gamma, \mathcal{S})) \to K_*(C^*_{max}(\Gamma, \mathcal{A}(E)))
$$
to be the homomorphism induced by the asymptotic morphism
$$
\beta_t:C^*_{max}(\Gamma, \mathcal{S}) \to C^*_{max}(\Gamma,\mathcal{A}(E))$$
given by $f \mapsto \beta_t(f) $ for each $t \in [1, \infty)$.
The following result is due to Higson--Kasparov--Trout \cite{KaHiTr}.

\begin{thm}[Infinite-dimensional Bott Periodicity \cite{KaHiTr}]
Let $\Gamma$ be a countable discrete group, $E$ an infinite-dimensional Euclidean space with a $\Gamma$-action by linear isometries. Then the Bott map
$$\beta_*:K_*(C^*_{max}(\Gamma, \mathcal{S})) \to K_*(C^*_{max}(\Gamma, \mathcal{A}(E)))$$
is an isomorphism.
\end{thm}

\subsection{$\Gamma$-$C^*$-algebras associated with coarse embeddings into Hilbert space}\label{sectn-proper-algebra-ce}
In the rest of this section, we shall define a proper $\Gamma$-$C^*$-algebra associated to a coarse embedding of $\Gamma$ into Hilbert space. Let us recall that a $\Gamma$-$C^*$-algebra $A$ is said to be proper if there exists a locally compact $\Gamma$-space $Y$ with a proper $\Gamma$-action such that $C_0(Y)$ is contained in the center of the multiplier algebra of $A$ and $C_0(Y)A$ is dense in $A$ under the norm topology.

In order to define the proper $\Gamma$-$C^*$-algebra, we will generalize the construction of Higson--Kasparov--Trout \cite{KaHiTr} to the case of continuous fields. The following construction is essentially due to Kasparov--Yu \cite{Kas_Yu},  Skandalis--Tu--Yu \cite{Skan_Tu_Yu}, and Tu \cite{Tu_BC}.

Suppose $\varphi: \Gamma \to H$ is a coarse embedding into Hilbert space. For each $\gamma \in \Gamma$, we define a bounded function $f_{\gamma}: \Gamma \to \mathbb{C}$ by
$$
f_{\gamma}(y)=\|\varphi(y)-\varphi(y\gamma)\|
$$
for all $y \in \Gamma$. The function $f_{\Gamma}$ is bounded since $\varphi$ is a coarse embedding.

Let $\ell^{\infty}(\gamma)$ be the $C^*$-algebra of all bounded complex-valued functions on $\Gamma$ and $c_0(\Gamma)\subset \ell^{\infty}(\Gamma)$ the $C^*$-subalgebra consisting of all functions vanishing at infinity. We define a $\Gamma$-action on $\ell^{\infty}(\Gamma)$ by $(\gamma\cdot f)(y)=f(y\gamma)$ for all $f \in \ell^{\infty}(\Gamma)$ and $x, \gamma \in \Gamma$.

Let $X'$ be the spectrum of the commutative $\Gamma$-invariant $C^*$-subalgebra of $\ell^{\infty}(\Gamma)$ generated by all constant functions, $c_0(\Gamma)$ functions and all functions $f_{\gamma}$ as defined above together with their translations by group elements of $\Gamma$. Then $X'$ admits a right action of $\Gamma$ induced by the $\Gamma$-action on $C(X')$ where $C(X')$ can be viewed as a $\Gamma$-invariant $C^*$-subalgebra of $\ell^{\infty}(\Gamma)$.

Note that $\Gamma$ is a dense subset of $X'$. For each $\gamma\in \Gamma$, the function $f_{\gamma}: \Gamma \to \mathbb{R}$ extends to a continuous function $\varphi'(\cdot, \gamma): X'\to \mathbb{C}$ by the definition of $X'$. One can define a continuous function $\varphi': X' \times \Gamma \to \mathbb{C}$ by continuously extending the function
$$\varphi'(y, \gamma)=f_{\gamma}(y)$$
for all $x, \gamma \in \Gamma$, where the space $X'\times \Gamma$ is equipped with product topology.

The continuous function $\varphi'$ on $X'\times \Gamma$ is a proper, continuous, conditionally negative definite function in the sense that it satisfies
\begin{enumerate}
    \item[(1)] $\varphi'(x,e)=0$ for all $x \in X$, where $e \in \Gamma$ is the identity element;
    \item[(2)] $\varphi'(xg,g^{-1})=\varphi(x,g)$ for all $x \in X$ and $g \in \Gamma$;
    \item[(3)] $\sum_{i=0}^n t_it_j\varphi'(xg_i,g_i^{-1}g_j)\leq 0$ for all $\left\{t_i\right\}_{i=1}^n\subset \mathbb{R}$ with $\sum_{i=1}^n=0$, $g_i\in \Gamma$ and $x \in X$;
    \item[(4)] $\varphi': X'\times \Gamma \to \mathbb{C}$ is proper in the sense that every preimage of a compact subset of $\mathbb{C}$ is compact.
\end{enumerate}
We say that the $\Gamma$-action on $X'$ is a-T-menable if there exists a proper, continuous, conditionally negative definite function on $X' \times \Gamma$.

Let $X$ be the space of probability measures on $X'$. It is a convex and compact topological space endowed with the weak-$*$ topology. The space $X$ admits a $\Gamma$-action induced by the $\Gamma$-action on $X'$. We define a continuous function on $X \times \Gamma$ by
$$
\varphi(m, \gamma)=\int_{X'} \varphi'(y,\gamma) dm(y)
$$
for all $m \in X$.

For each pair $(m,g) \in X\times \Gamma$, we have that
\begin{align*}
    \varphi(mg,g^{-1}) &=\int_{X'} \varphi'(y,\gamma^{-1}) d(mg)\\
                       &=\int_{X'} \varphi'(yg,\gamma^{-1}) dm\\
                       &=\int_{X'} \varphi'(y,\gamma) dm(y)\\
                       &=\varphi(m,g).
\end{align*}
Note that $\varphi(x,e)=0$ for all $x \in X$. By the definition of the function $\varphi$ and the properties of $\varphi'$, we have that the continuous function $\varphi$ is a proper, and conditionally negative definite function. Note that the $\Gamma$-space $X$ satisfies the following
\begin{enumerate}
    \item[(1)] for each finite subgroup $F\subseteq G$, $X$ is $F$-contractible;
    \item[(2)] the $\Gamma$-action is a-T-menable.
\end{enumerate}

Now, we are ready to construct a continuous field of Hilbert spaces using the action of $\Gamma$ on the space $X$. Let us first recall the definition of continuous fields of Hilbert spaces over a compact space. Let $\big(\mathcal{H}_x\big)_{x \in X}$ be a family of Banach spaces. Denote $\mathcal{H}=\bigsqcup_{x \in X}\mathcal{H}_{x}$. A section of the bundle $\mathcal{H}$ is a function $s: X \to \mathcal{H}$ satisfying $s(x) \in \mathcal{H}_{x}$ for all $x \in X$.

\begin{defn}
Let $X$ be
a compact space. A continuous field of Banach spaces over $X$ is a family of Banach spaces $\big(\mathcal{H}_x\big)_{x \in X}$ with a set of sections $\Theta(X, \mathcal{H})$, such that
\begin{enumerate}
  \item[(1)] the set $\Theta(X, \mathcal{H})$ is a linear subspace of the direct product $\prod_{x \in X} \mathcal{H}_x$:
  \item[(2)] for every $x \in X$, the set $\{s(x):s \in \Theta(X, \mathcal{H})\}$ is dense in $\mathcal{H}_x$;
  \item[(3)] for every $s \in \Theta(X, \mathcal{H})$, the function $x \mapsto \|s(x)\|$ is a continuous function on $X$;
  \item[(4)] let $s: X \to \mathcal{H}$ be a section, i.e. $s(x) \in \mathcal{H}_{x}$, for all $x \in X$. If for every $x \in X$, and every $\epsilon>0$, there exists a section $s' \in \Theta(X, \mathcal{H})$ such that $\|s(y)-s'(y)\|< \epsilon$ for all $y$ in some neighborhood of $x$, then $s \in \Theta(X, \mathcal{H})$.
\end{enumerate}
\end{defn}

If every fiber $\mathcal{H}_x$ is a Hilbert space, we say $\big(\mathcal{H}_x\big)_{x \in X}$ is a continuous field of Hilbert spaces over $X$. If every fiber is a $C^*$-algebra and the collection of sections is closed under the $*$-operation and the multiplication, the continuous field is called a continuous field of $C^*$-algebras.

Let $\varphi: X \times \Gamma \to \mathbb{R}$ be a continuous, proper conditionally negative definite
function. We can define a continuous field of Hilbert spaces as follows.

Consider a linear subspace
$$C^0_{c}(\Gamma):=\left\{f \in C_c(\Gamma): \sum_{g \in \Gamma}f(g)=0\right\}\subset C_c(\Gamma).$$
 For each $x \in X$, we define a sesqui-linear form
$$
\left\langle\xi,\eta\right\rangle_x=-\frac{1}{2}\sum_{g, g' \in \Gamma}\xi(g) \overline{\eta(g')}\varphi(xg^{-1}, gg'),
$$
for all $\xi, \eta \in C^0_{c}(\Gamma)$. Since $\varphi$ is of conditionally negative definite type, the form above is positive semidefinite and so one can quotient out by the zero subspace and complete to get a Hilbert space $\mathcal{H}_x$. Following the arguments in \cite{Deng-Novikov-extension}, we have a continuous field of Hilbert spaces $\left(\mathcal{H}_x\right)_{x\in X}$.

Since each fiber of the continuous field is a Hilbert space, we can define a $C^*$-algebra $\mathcal{A}(\mathcal{H}_x)$ associated with each fiber $\mathcal{H}_x$ following the construction in Section \ref{section-C-alg-for-Hilbert-space}. Furthermore, by the first author's construction in \cite{Deng-Novikov-extension}, one obtains a $C^*$-algebra with proper $\Gamma$-action.
\begin{thm}[\cite{Deng-Novikov-extension}]
Let $\big(\mathcal{A}(\mathcal{H}_x)\big)_{x \in X}$ be the collection of $C^*$-algebras defined above.
\begin{enumerate}
    \item[(1)] There exists a structure of a continuous field of $C^*$-algebras for the bundle $\big(\mathcal{A}(\mathcal{H}_x)\big)_{x \in X}$.
     \item[(2)] Let $\mathcal{A}(X)$ be the $C^*$-algebra generated by all the continuous sections over the continuous field. Then there exists a proper $\Gamma$-action on the $\mathcal{A}(X)$.
\end{enumerate}
\end{thm}

We also define a $G$-action on $\mathcal{A}(X)$ by
$$g\cdot a=h(g)\cdot a$$
for all $g \in G$ and $a \in \mathcal{A}(X)$. Then we obtain a $G$-$\Gamma$-$C^*$-algebra $\mathcal{A}(X)$. We can view $\mathcal{S}$ as a $G$-$\Gamma$-algebra with trivial $G$-action and $\Gamma$-action.

Next, we shall discuss about the $K$-theory of $\mathcal{A}(X)$. Indeed, the computation of its $K$-theory plays a crucial role in the proof of the relative Novikov conjecture.

For each $x \in X$, we have the asymptotic morphism
$$\beta_{x,t}:\mathcal{S}\to \mathcal{A}(\mathcal{H}_x),$$
for $t \in [1,\infty)$. Accordingly, we have an asymptotic morphism
$$\beta_t:\mathcal{S}\to \mathcal{A}(X)$$
defined by
$$
\beta_t(f)(x)=\beta_{x,t}(f)
$$
for all $f \in \mathcal{S}$ and $t \in [1,\infty)$. Following the arguments in \cite{KaHiTr}, we can define asymptotic morphisms
$$\beta_t: C^*_{red}(\Gamma, \mathcal{S})\to C^*_{red}(\Gamma,\mathcal{A}(X))$$
and
$$\beta_t: C^*_{max}(\Gamma, \mathcal{S})\to C^*_{max}(\Gamma,\mathcal{A}(X)),$$
for all $t \in [1, \infty)$.

In order to define the asymptotic morphisms between localization algebras, we shall define the asymptotic morphisms between Roe algebras. For each element $T=\left(T_{x,y}\right)_{x,y\in Z_s}\in \mathbb{C}[P_s(G)]^G\widehat{\otimes} \mathcal{S}$, we define a $Z_s$-by-$Z_s$-matrix
$$
\left(\beta_t(T)\right)_{x,y}=T_{x,y}\widehat{\otimes} \beta_{x,t}(f)
$$
for each $t \in [1,\infty)$ and all $x, y \in Z_s$. It is obvious that $\beta_t(T)$ is an element in $\mathbb{C}[P_s(G), \mathcal{A}(X)]^G$. As a result, we can define asymptotic morphisms
$$
\beta_t: \mathbb{C}[ P_s(G)]^G\widehat{\otimes} \mathcal{S}\to \mathbb{C}[P_s(G), \mathcal{A}(X)]^G.
$$
and
$$
\beta_{t}: \mathbb{C}[P_s(\Gamma)]^{\Gamma}\widehat{\otimes} \mathcal{S}\to \mathbb{C}[P_s(\Gamma), \mathcal{A}(X)]^{\Gamma},
$$
for all $t \in [1, \infty)$.
Similarly, we define asymptotic morphisms
$$
\beta_{L,t}: \mathbb{C}_{max,L}[P_s(G)]^G\widehat{\otimes} \mathcal{S}\to \mathbb{C}_{max,L}[P_s(G), \mathcal{A}(X)]^G.
$$
and
$$
\beta_{L,t}: \mathbb{C}_{max,L}[P_s(\Gamma)]^{\Gamma}\widehat{\otimes}\mathcal{S}\to \mathbb{C}_{max,L}[P_s(\Gamma), \mathcal{A}(X)]^{\Gamma},
$$
for $t \in [1, \infty)$.
Moreover, the above asymptotic morphism between algebraic Roe algebras and localization algebras induce the following asymptotic morphisms:
\begin{enumerate}
     \item[(1)] $\beta_t: C_{max}^*(P_s(G))^G\widehat{\otimes}\mathcal{S} \to C_{max}^*(P_s(G), \mathcal{A}(X))^G$;
      \item[(2)] $\beta_t: C_{max}^*(P_s(\Gamma))^{\Gamma}\widehat{\otimes}\mathcal{S} \to C_{max}^*(P_s(\Gamma), \mathcal{A}(X))^{\Gamma}$;
       \item[(3)] $\beta_{L,t}: C_{max,L}^*(P_s(G))^G \widehat{\otimes}\mathcal{S} \to C_{max,L}^*(P_s(G), \mathcal{A}(X))^G$;
      \item[(4)] $\beta_{L,t}: C^*_{max, L}(P_s(\Gamma))^{\Gamma}\widehat{\otimes}\mathcal{S}\to C^*_{max,L}(P_s(\Gamma), \mathcal{A}(X))^{\Gamma}$,
\end{enumerate}
for all $t \in [1, \infty)$.
%========================================================================================================
\begin{comment}
    \item[(1)] $\beta: C^*(G, P_s(G), A) \widehat{\otimes}\mathcal{S} \to C^*(G, P_s(G), \mathcal{A}(X))$;
    \item[(2)] $\beta: C^*(\Gamma, P_s(\Gamma), A)\widehat{\otimes}\mathcal{S}\to C^*(\Gamma, P_s(\Gamma), \mathcal{A}(X))$;
\end{comment}
%=======================================================================================================

%Following the constructions in \cite{Deng-Novikov-extension}, w
\begin{comment}
We have the following Bott periodicity.
\begin{thm}[\cite{Deng-Novikov-extension}] \label{thm-Bott-for-finite-subgroup}
For each finite subgroup $F \subseteq \Gamma$, the homomorphism
    $$\beta_*: K_*(\mathcal{S}\rtimes F)\to K_*(\mathcal{A}(X)\rtimes F)$$
    induced by the asymptotic morphism $\beta_t: \mathcal{S}\rtimes F \to \mathcal{A}(X)\rtimes F$ on $K$-theory is an isomorphism.
\end{thm}

 As consequence, w
 \end{comment}

 Since the group actions of $G$ and $\Gamma$ on $\mathcal{S}$ are trivial, we have that
 $$C_{max}^*(P_s(G))^G\widehat{\otimes}\mathcal{S} \cong C^*_{max}(G, \mathcal{S}),$$
 $$C_{max}^*(P_s(\Gamma))^{\Gamma}\widehat{\otimes}\mathcal{S} \cong C^*_{max}(\Gamma, \mathcal{S}),$$
  $$C_{max, L}^*(P_s(G))^G\widehat{\otimes}\mathcal{S} \cong C^*_{max,L}(G, \mathcal{S}),$$
 and
  $$C_{max, L}^*(P_s(\Gamma))^{\Gamma}\widehat{\otimes}\mathcal{S} \cong C^*_{max, L}(\Gamma, \mathcal{S}).$$

 As a consequence of the Mayer--Vietoris sequence and the five lemma, we have the following Bott periodicity.
\begin{prop}
For each $s>0$, the maps
$$\beta_{L,*}: K_*(C_{max,L}^*(P_s(G), \mathcal{S})^G) \to K_*(C^*_{max,L}(P_s(G), \mathcal{A}(X))^G)$$
and
$$\beta_{L,*}: K_*(C_{max,L}^*(P_s(\Gamma), \mathcal{S})^{\Gamma}) \to K_*(C^*_{max,L}(P_s(\Gamma), \mathcal{A}(X))^{\Gamma})$$
induced by the asymptotic morphisms $(\beta_{L,t})_{t \in [1, \infty)}$ on $K$-theory are isomorphisms.
\end{prop}
\begin{comment}
\begin{proof}
Using the strongly Lipschitz homotopy equivalence of the $K$-theory of localization algebras and the Mayer--Vietoris sequence for the $K$-theory of localization algebras (cf. \cite{Yu_localization}), and the five lemma.
\end{proof}
\end{comment}
%the proposition follows from Theorem \ref{thm-Bott-for-finite-subgroup} and
%By the constructions of the homomorphisms between maximal Roe algebras in Section \ref{section-relative-Roe}, we have the asymptotic morphisms
%$$\beta_t \colon C_{max}^*(P_s(G),\mathcal{S})^G \to C_{max}^*(P_s(G),\mathcal{A}(X))^G$$
%and
%$$\beta_t \colon C_{max}^*(P_s(\Gamma),\mathcal{S})^{\Gamma} \to  C_{max}^*(P_s(\Gamma),\mathcal{A}(X))^{\Gamma}$$
%for all $t \in [1, \infty)$.

By the definition of the above asymptotic morphisms, the following diagram
\begin{equation*}
\begin{tikzcd}
    C_{max}^*(P_s(G),\mathcal{S})^G \ar[r, "\beta_t"]\ar[d]& C_{max}^*(P_s(G),\mathcal{A}(X))^G\ar[d]\\
    C_{max}^*(P_s(\Gamma),\mathcal{S})^{\Gamma} \ar[r, "\beta_t"] &C_{max}^*(P_s(\Gamma),\mathcal{A}(X))^{\Gamma}
\end{tikzcd}
\end{equation*}
asymptotically commutes. As a result, we obtain an asymptotic morphism between relative Roe algebras
$$
\beta_t: C^*_{max}(P_s(G), P_s(\Gamma),  \mathcal{S})^{G, \Gamma}\to C^*_{max}(P_s(G), P_s(\Gamma), \mathcal{A}(X))^{G,\Gamma}
$$
for all $t \in [1,\infty)$ and $s>0$. Similarly, the asymptotic morphism $\left(\beta_t: \mathcal{S} \to \mathcal{A}(X)\right)_{t \in [1,\infty)}$ also induces an asymptotic morphism between relative localization algebras
$$\beta_{L,t}: C_{max,L}^*(P_s(G), P_{s}(\Gamma), \mathcal{S})^{G, \Gamma} \to  C_{max,L}^*(P_s(G), P_{s}(\Gamma),\mathcal{A}(X))^{G, \Gamma},$$
for $t \in [1,\infty)$.
Therefore, we have a homomorphism
$$\beta_{L,*}: K_*(C_{max,L}^*(P_s(G),P_{s}(\Gamma), \mathcal{S})^{G, \Gamma})\to   K_*(C_{max,L}^*(P_s(G),P_{s}(\Gamma),\mathcal{A}(X))^{G, \Gamma})$$
induced by the asymptotic morphism above on $K$-theory. Passing to inductive limits, we have the relative Bott map
$$\beta_{L, *}^{G, \Gamma}\colon  K_{*}^{G,\Gamma}(\underline{E}G,\underline{E}\Gamma,\mathcal{S})\rightarrow K_{*}^{G,\Gamma}(\underline{E}G,\underline{E}\Gamma,\mathcal{A}(X)).$$

We shall prove the following relative Bott periodicity.
\begin{prop}[Relative Bott Periodicity]\label{prop-rel-Bott}
The relative Bott map
$$\beta_{L, *}^{G, \Gamma}\colon  K_{*}^{G,\Gamma}(\underline{E}G,\underline{E}\Gamma,\mathcal{S})\rightarrow K_{*}^{G,\Gamma}(\underline{E}G,\underline{E}\Gamma,\mathcal{A}(X))$$
induced by the asymptotic morphism between relative localization algebras is an isomorphism.
\end{prop}
\begin{proof}
We have the following commutative diagram:

\begin{equation*}
    \begin{tikzcd}
   K_{*+1}^{G}(\underline{E}G,\mathcal{S}) \ar{r}{\beta_{L,*}^G}[swap]{\cong}\ar{d}
   &K_{*+1}^{G}(\underline{E}G,\mathcal{A}(X)) \ar{d}\\
K_{*+1}^{\Gamma}(\underline{E}\Gamma,\mathcal{S})  \ar{r}{\beta_{L,*}^\Gamma}[swap]{\cong}\ar{d}
&  K_{*+1}^{\Gamma}(\underline{E}\Gamma,\mathcal{A}(X))\ar{d}\\
K_{*+1}^{G,\Gamma}(\underline{E}G,\underline{E}\Gamma,\mathcal{S})\ar{r}{\beta_{L, *}^{G, \Gamma}} \ar{d}
& K_{*+1}^{G,\Gamma}(\underline{E}G,\underline{E}\Gamma,\mathcal{A}(X))\ar{d}\\
  K_{*}^{G}(\underline{E}G,\mathcal{S}) \ar{d}\ar{r}{ \beta_{L,*}^G}[swap]{\cong}
  &    K_{*}^{G}(\underline{E}G,\mathcal{A}(X))  \ar{d}\\
K_{*}^{\Gamma}(\underline{E}\Gamma,\mathcal{S}) \ar{r}{\beta_{L,*}^\Gamma}[swap]{\cong}
&  K_{*}^{\Gamma}(\underline{E}\Gamma,\mathcal{A}(X)).
\end{tikzcd}
\end{equation*}

Since the Bott maps $\beta_{L,*}^G$ and $\beta_{L, *}^{\Gamma}$ are isomorphic, it follows from the five lemma that the map
$$\beta_{L, *}^{G, \Gamma}\colon  K_{*}^{G,\Gamma}(\underline{E}G,\underline{E}\Gamma,\mathcal{S})\rightarrow K_{*}^{G,\Gamma}(\underline{E}G,\underline{E}\Gamma,\mathcal{A}(X))$$
is an isomorphism.

\end{proof}

\section{The proofs of the main results}
In this section, we shall prove the maximal strong relative Novikov conjecture  and the reduced strong relative Novikov conjecture with coefficients in a ${\rm II}_1$-factor for a pair of groups $(G, \Gamma)$ under certain assumptions on the geometry of $\Gamma$ and the kernel of the homomorphism $h: G \to \Gamma$.

\subsection{The maximal strong relative Novikov conjecture}
Let us first prove the maximal strong relative Novikov conjecture for the following pairs of groups $(G, \Gamma)$.
\begin{thm}\label{thm-maxi-Novikov}
Let $G$ and $\Gamma$ be finitely generated groups and $h: G \to \Gamma$ a group homomorphism with  the maximal good kernel property (cf. Definition \ref{weakgamma}). Assume that $\Gamma$ admits a coarse embedding into Hilbert space. Then the maximal strong Novikov conjecture holds for $(G, \Gamma)$, i.e. the maximal relative assembly map
$$\mu_{max}:K_{*}^{G,\Gamma}(\underline{E}G,\underline{E}\Gamma)\to K_{*}(C_{max}^*(G,\Gamma))$$
is injective.
\end{thm}
In particular, if the group $\Gamma$ admits a coarse embedding into Hilbert space and $\ker(h)$ is a-T-menable, it follows from Theorem \ref{thm-maxi-Novikov} that the maximal strong Novikov conjecture holds for $(G, \Gamma)$, i.e. the maximal relative assembly map
$$\mu_{max}:K_{*}^{G,\Gamma}(\underline{E}G,\underline{E}\Gamma)\to K_{*}(C_{max}^*(G,\Gamma))$$
is an isomorphism. The special case where  both $G$ and $\Gamma$ are a-T-menable was proved by the second author in  \cite{tian-thesis}.

\begin{proof}[Proof of Theorem \ref{thm-maxi-Novikov}]
For each $s>0$, we have the asymptotically commutative diagram:
\begin{equation*}
\begin{tikzcd}
    C_{L}^*(P_s(G), P_{s}(\Gamma),\mathcal{S})^{G, \Gamma} \ar[r]\ar[d]& C_{max}^*(G,\Gamma,\mathcal{S})\ar[d]\\
    C_{L}^*(P_s(G), P_{s}(\Gamma),\mathcal{A}(X))^{G, \Gamma} \ar[r] &C_{max}^*(G,\Gamma,\mathcal{A}(X)).
\end{tikzcd}
\end{equation*}
Passing to inductive limits gives rise to the following commutative diagram:
\begin{equation*}
   \begin{tikzcd}
    K_{*}^{G,\Gamma}(\underline{E}G,\underline{E}\Gamma,\mathcal{S})\ar{r}{\mu_{max}}\ar{d}{\beta_{L,*}^{G, \Gamma}}
    & K_{*}(C_{max}^*(G,\Gamma,\mathcal{S})) \ar{d}{\beta_*}\\
   K_{*}^{G,\Gamma}(\underline{E}G,\underline{E}\Gamma,\mathcal{A}(X))\ar{r}{\mu^{\mathcal{A}(X)}_{max}}&
   K_{*}(C_{max}^*(G,\Gamma, \mathcal{A}(X))).
   \end{tikzcd}
\end{equation*}
Since $\beta_{L, *}^{G, \Gamma}$ is an isomorphism by Proposition \ref{prop-rel-Bott}, it suffice to show that the relative assembly map with coefficients in $\mathcal{A}(X)$
$$\mu^{\mathcal{A}(X)}_{max}:K_{*}^{G,\Gamma}(\underline{E}G,\underline{E}\Gamma,\mathcal{A}(X))\to
   K_{*}(C_{max}^*(G,\Gamma, \mathcal{A}(X)))$$
   is an isomorphism.

Consider the commutative diagram:
\begin{equation*}
    \xymatrix@C+3pc{
    K_{*+1}^{G}(\underline{E}G,\mathcal{A}(X)) \ar[d]\ar[r]^{\mu_{G}^{\mathcal{A}(X)}}
    &  K_{*+1}(C_{max}^*(G,\mathcal{A}(X))) \ar[d]\\
K_{*+1}^{\Gamma}(\underline{E}\Gamma,\mathcal{A}(X))\ar[r]^{\mu_{\Gamma}^{\mathcal{A}(X)}}\ar[d]&  K_{*+1}(C_{max}^*(\Gamma,\mathcal{A}(X)))  \ar[d]\\
K_{*+1 }^{G,\Gamma}(\underline{E}G,\underline{E}\Gamma,\mathcal{A}(X))\ar[r]^{\mu_{max}^{\mathcal{A}(X)}}\ar[d]
& K_{*+1}(C_{max}^*(G,\Gamma, \mathcal{A}(X)))\ar[d]\\
  K_{*}^{G}(\underline{E}G,\mathcal{A}(X)) \ar[r]^{\mu_{G}^{\mathcal{A}(X)}}\ar[d]
  &  K_{*}(C_{max}^*(G,\mathcal{A}(X)))\ar[d] \\
K_{*}^{\Gamma}(\underline{E}\Gamma,\mathcal{A}(X)) \ar[r]^{\mu_{\Gamma}^{\mathcal{A}(X)}} &  K_{*}(C_{max}^*(\Gamma,\mathcal{A}(X))).
    }
\end{equation*}
Since $\mathcal{A}(X)$ is a proper $\Gamma$-$C^*$-algebra, then we have that the assembly map
$$
\mu_{\Gamma}^{\mathcal{A}(X)}:K^{\Gamma}_*(\underline{E}\Gamma, \mathcal{A}(X)) \to K_*(C^*_{max}(\Gamma, \mathcal{A}(X)))
$$
is an isomorphism. We will use the five lemma to prove that the assembly map $\mu_{max}^{\mathcal{A}(X)}$ is an isomorphism. For this purpose, we shall show that the assembly map
$$
\mu_{G}^{\mathcal{A}(X)}: K^G_*(\underline{E}G, \mathcal{A}(X)) \to K_*(C^*_{max}(G, \mathcal{A}(X)))
$$
is an isomorphism. We remark here that although the action of $G$ on $\mathcal{A}(X)$ is not proper, we can prove that the assembly map $\mu_{G}^{\mathcal{A}(X)}$ is an isomorphism using the cutting-and-pasting method.

Since $\mathcal{A}(X)$  is a proper $\Gamma$-algebra, it is also an $h(G)$-proper algebra. We can express the algebra as a direct limit
$$\mathcal{A}(X)=\lim\limits_{\longrightarrow}\mathcal{A}_\alpha,$$
where  each $\mathcal{A}_\alpha$ is an ideal of $\mathcal{A}(X)$, and a proper $h(G)$-$C^*$-algebra over a proper, cocompact and locally compact Hausdorff $h(G)$-space $W_\alpha$. It suffices to prove that the maximal assembly map
$$\mu_{max}: K_{*}^{G}(\underline{E}G,\mathcal{A}_\alpha) \stackrel{  }{\longrightarrow} K_{*}(C_{max}^*(G,\mathcal{A}_\alpha)) $$
is an isomorphism.
%it suffices to show
%$$\mu_{max}: K_{*}^{G}(\underline{E}G,B) \stackrel{  }{\longrightarrow} K_{*}(C_{max}^*(G,B)) $$ where

 Let $F \subset h(G)$ be a finite subgroup, and $Y$ a $F$-space. Denote $Y\times_F h(G)$ to be the quotient space of the product space $Y \times h(G)$ over the $F$-action by $\gamma \cdot (y, g)=(\gamma y, \gamma g)$ for all $\gamma \in F$, $(y, g) \in Y \times h(G)$. If $Y\subset W_{\alpha}$ is a $F$-invariant subset such that $gY\cap Y=\emptyset $ for each $g \notin F$, then we can view $Y\times_{F}h(G)$ as the subset $h(G)\cdot Y$ of $Y$ via the map $[(y, g)] \mapsto gy$ for all $[(y, g)] \in Y\times_F h(G)$.

 Since the locally compact space $W_{\alpha}$ is $h(G)$-proper and cocompact, it is a finite union of the form
$$W_{\alpha} = \bigcup_i^{n} Y_i \times_{F_i} h(G),$$
where each $F_i$ is a finite subgroup of $h(G)$, and $Y_i$ is a  precompact $F_i$-space for $1 \leq i \leq n$.

For each $i$, denote by $B=C_0(Y_i\times_{F_i}h(G))\cdot \mathcal{A}_{\alpha}$ the proper $h(G)$-$C^*$-subalgebra of $\mathcal{A}_{\alpha}$. Let $B_0=C_0(Y_i)\cdot \mathcal{A}_{\alpha}$.  The $C^*$-algebra $B$ is equipped with a $G$-action by lifting the $h(G)$-action on $B$ and $B_0$ is equipped with an $h^{-1}(F_i)$-action by lifting the $F_i$-action on $B_0$. Note that
$$
C^*_{max}(G, B)\cong C^*_{max}( h^{-1}(F_i), B_0) \otimes K,
$$
where $K$ is the algebra of compact operators.
 As a result, we have the following commutative diagram
\begin{equation*}
  \begin{tikzcd}
    K_{*}^{G}(\underline{E}G,B) \ar{d}{\cong}\ar[r] &K_{*}(C_{max}^*(G,B)) \ar{d}{\cong}\\
K_{*}^{h^{-1}(F_i)}(\underline{E}(h^{-1}(F_i)),B_0)\ar[r] &K_{*}(C_{max}^*(h^{-1}(F_i),B_0)).
  \end{tikzcd}
\end{equation*}
Since $h: G\to \Gamma$ has the maximal good kernel property, and $[h^{-1}(F_i): \ker(h)]<\infty$, the bottom map is isomorphic. Therefore, the assembly map
$$
\mu: K^G_*(\underline{E}G, B) \to K_*(C^*_{max}(G, B))
$$
is an isomorphism. It follows from   the Mayer--Vietoris sequence and the five lemma
that the assembly map
$$\mu_{G}^{\mathcal{A}_{\alpha}}: K^G_*(\underline{E}G, \mathcal{A}_{\alpha})\to K_*(C^*_{max}(G, \mathcal{A}_{\alpha})$$
is an isomorphism. Passing to the inductive limit, we have that the assembly map
$$\mu_{G}^{\mathcal{A}(X)}: K^G_*(\underline{E}G, \mathcal{A}(X))\to K_*(C^*_{max}(G, \mathcal{A}(X))$$ is isomorphic. As a result,
$$\mu^{\mathcal{A}(X)}_{max}:K^{G, \Gamma}_*(\underline{E}G, \underline{E}\Gamma, \mathcal{A}(X))\to K_*(C^*_{max}(G, \Gamma, \mathcal{A}(X)))$$
is an isomorphism.
Therefore, the relative assembly map
$$\mu_{max}:K_{*}^{G,\Gamma}(\underline{E}G,\underline{E}\Gamma)\to
   K_{*}(C_{max}^*(G,\Gamma))$$
   is injective. This finishes the proof.

\end{proof}

\subsection{The reduced strong relative Novikov conjecture}
In this subsection, we shall prove the reduced strong  relative Novikov conjecture with coefficients in a ${\rm II}_1$-factor $\mathcal{M}$ for the following pairs of groups $(G, \Gamma)$.

\begin{thm}\label{thm-reduced-Novi-real-coef}
Let $h\colon  G \to \Gamma$ be a group homomorphism with the reduced good kernel property (cf. Definition \ref{weakgamma}). Let $\phi: C^*_{red}(G) \hookrightarrow \mathcal M$ be a trace-preserving embedding. Assume that $\Gamma$ admits a coarse embedding into Hilbert space. Then the reduced strong relative Novikov conjecture holds for $h\colon G\to \Gamma$, i.e.,  the reduced  relative assembly map
$$\mu_{red}: K^{G, \Gamma}_*(\underline{E}G, \underline{E}\Gamma, \mathcal M) \to K_*(C^*_{red}(G, \Gamma, \mathcal M))$$
is injective.
\end{thm}
As an example, when the group $\Gamma$ is coarsely embeddable into Hilbert space and $\ker(h)$ is a subgroup of a hyperbolic group, the reduced strong relative Novikov conjecture holds for $h\colon G\to \Gamma$.

The proof of Theorem \ref{thm-reduced-Novi-real-coef} is similar to that of Theorem \ref{thm-maxi-Novikov}. Let $\mathcal{A}(X)$ be the proper $\Gamma$-algebra defined in Section \ref{sectn-proper-algebra-ce}. We have the Bott asymptotic morphism
$$\beta_t: \mathcal{S}\to \mathcal{A}(X),$$
for all $t\in [1,\infty)$.
It induces asymptotic morphisms
$$\beta_t: C_{red}^*(G,\Gamma,\mathcal{S}\otimes \mathcal M)\to  C_{red}^*(G,\Gamma,\mathcal{A}(X)\otimes\mathcal  M)$$
and
$$\beta_t: C_{L}^*(P_s(G),P_{s}(\Gamma),\mathcal{S}\otimes \mathcal M)^{G, \Gamma} \to  C_{L}^*(P_s(G), P_{s}(\Gamma), \mathcal{A}(X)\otimes \mathcal M)^{G, \Gamma},$$
for all $t \in [1, \infty)$. The latter induces the following Bott map
$$\beta_{L, *}^{G, \Gamma}: K_{*}^{G,\Gamma}(\underline{E}G,\underline{E}\Gamma,\mathcal{S}\otimes \mathcal M)\rightarrow K_{*}^{G,\Gamma}(\underline{E}G,\underline{E}\Gamma,\mathcal{A}(X)\otimes \mathcal M).$$

\begin{prop} The Bott map
$$\beta_{L, *}^{G, \Gamma}: K_{*}^{G,\Gamma}(\underline{E}G,\underline{E}\Gamma,\mathcal{S}\otimes \mathcal M)\rightarrow K_{*}^{G,\Gamma}(\underline{E}G,\underline{E}\Gamma,\mathcal{A}(X)\otimes \mathcal M)$$ is an isomorphism.
\end{prop}
Since the proof of the above result follows from the same arguments in its maximal analogue (Proposition \ref{prop-rel-Bott}), we omit its proof.
%=====================================================================================================

\begin{comment}
\begin{proof} We have the following commutative diagram:
\begin{equation*}
\begin{tikzcd}
   K_{*+1}^{G}(\underline{E}G,\mathcal{S}) \ar[d]\ar{r}{\beta^G_L}&    K_{*+1}^{G}(\underline{E}G,\mathcal{A}(X)) \ar[d]\\
K_{*+1}^{\Gamma}(\underline{E}\Gamma,\mathcal{S}) \otimes\mathbb{R}\ar[d]\ar{r}{\beta^{\Gamma}_L}&  K_{*+1}^{\Gamma}(\underline{E}\Gamma,\mathcal{A}(X)) \otimes\mathbb{R} \ar[d]\\
K_{*,\mathbb{R}}^{G,\Gamma}(\underline{E}G,\underline{E}\Gamma,\mathcal{S})\ar{r}{\beta_L}\ar[d]
& K_{*,\mathbb{R}}^{G,\Gamma}(\underline{E}G,\underline{E}\Gamma,\mathcal{A}(X))\ar[d]\\
  K_{*}^{G}(\underline{E}G,\mathcal{S}) \ar{r}{\beta^G_L}\ar[d]&    K_{*}^{G}(\underline{E}G,\mathcal{A}(X)) \ar[d] \\
K_{*}^{\Gamma}(\underline{E}\Gamma,\mathcal{S}) \otimes\mathbb{R}\ar{r}{\beta^{\Gamma}_L}&  K_{*}^{\Gamma}(\underline{E}\Gamma,\mathcal{A}(X)) \otimes\mathbb{R}
 \end{tikzcd}
\end{equation*}
\end{proof}
\end{comment}

%=====================================================================================================
\begin{proof}[Proof of Theorem \ref{thm-reduced-Novi-real-coef}]

Consider the commutative diagram:
\begin{equation*}
    \begin{tikzcd}
    K_{*}^{G,\Gamma}(\underline{E}G,\underline{E}\Gamma,\mathcal{S}\otimes \mathcal M)\ar{d}{\beta_{L,*}^{G, \Gamma}}\ar{r}{\mu_{red}}
    & K_{*}(C_{red}^*(G,\Gamma,\mathcal{S}\otimes \mathcal M))\ar{d}{\beta_*}\\
    K_{*}^{G,\Gamma}(\underline{E}G,\underline{E}\Gamma,\mathcal{A}(X)\otimes \mathcal M) \ar{r}{\mu^{\mathcal{A}(X)}_{red}}&
    K_{*}(C_{red}^*(G,\Gamma, \mathcal{A}(X)\otimes \mathcal M)).
    \end{tikzcd}
\end{equation*}
Since $\beta_{L,*}^{G, \Gamma}$ is an isomorphism, to prove that the relative assembly map $\mu_{red}$ is injective, it suffices to show that $\mu^{\mathcal{A}(X)}_{red}$ is isomorphic. Indeed, it can be proved by the same cutting-and-pasting method used in the proof of Theorem \ref{thm-maxi-Novikov}. We have the following commutative diagram:
 \begin{equation*}
 \begin{tikzcd}
  K_{*+1}^{G}(\underline{E}G,\mathcal{A}(X)\otimes \mathcal M)\ar{r}{\mu_G^{\mathcal{A}(X)}} \ar{d} &  K_{*+1}(C_{red}^*(G,\mathcal{A}(X)\otimes \mathcal M))\ar{d} \\
K_{*+1}^{\Gamma}(\underline{E}\Gamma,\mathcal{A}(X)\otimes \mathcal M\otimes \mathcal M) \ar{r}{\mu_{\Gamma}^{\mathcal{A}(X)}}\ar{d} & K_{*+1}(C_{red}^*(\Gamma,\mathcal{A}(X)\otimes \mathcal M\otimes \mathcal M)) \ar{d} \\
K_{*+1}^{G,\Gamma}(\underline{E}G,\underline{E}\Gamma,\mathcal{A}(X)\otimes \mathcal M)\ar{d}\ar{r}{\mu^{\mathcal{A}(X)}_{red}}
& K_{*+1}(C_{red}^*(G,\Gamma, \mathcal{A}(X)\otimes \mathcal M))\ar{d} \\
  K_{*}^{G}(\underline{E}G,\mathcal{A}(X)\otimes \mathcal M) \ar{r}{\mu^{\mathcal{A}(X)}_{G}}\ar{d}&  K_{*}(C_{red}^*(G,\mathcal{A}(X)\otimes \mathcal M)) \ar{d}\\
K_{*}^{\Gamma}(\underline{E}\Gamma,\mathcal{A}(X)\otimes\mathcal  M\otimes \mathcal M)\ar{r}{\mu_{\Gamma}^{\mathcal{A}(X)}}&  K_{*}(C_{red}^*(\Gamma,\mathcal{A}(X)\otimes \mathcal M\otimes \mathcal M)).
 \end{tikzcd}
 \end{equation*}
By the five lemma, we have that $\mu_{red}^{\mathcal{A}(X)}$ is an isomorphism. This finishes the proof.
\end{proof}

\section{Applications to the relative Novikov conjecture}\label{sec:app}

In this section, we shall discuss an application of the  maximal (reduced) strong relative Novikov conjecture to the relative Novikov conjecture regarding the homotopy invariance of relative higher signatures of manifolds with boundary. We shall first construct the relative higher indices of signature operators on compact manifolds with boundary. Then we will show that the relative higher indices for signature operators  are invariant under orientation-preserving homotopy equivalences of manifolds with boundary as pairs.

Let $M$ be an oriented compact smooth manifold with boundary $\partial M$.  Suppose  $\pi_1(\partial M)=G$ and $\pi_1(M)=\Gamma$. Let
$$h: G\rightarrow \Gamma$$
be the natural homomorphism  induced by the inclusion of $\partial M$ into $M$.

%Let $M$ be a compact oriented manifold with boundary $\partial M$, and let $G=\pi_1 \partial M$ and $\Gamma=\pi_1 M$ denote their fundamental groups. We define a relative assembly map
%$$
%\mu_{max}: K^{G,\Gamma}_*(\underline{E}G,\underline{E}\Gamma) \to K_*(C^*_{max}(G, \Gamma)),
%$$
%where $K^{G,\Gamma}_*(\underline{E}G,\underline{E}\Gamma)$ is the relative $K$-homology and $C^*_{max}(G,\Gamma)$ is the maximal relative group $C^*$-algebra associated with the group homomorphism $h: G \to \Gamma$ induced by the inclusion of $\partial M$ into $M$. In this paper, we study the injectivity of the assembly map under assumptions on the geometry of the groups $\Gamma$ and $\ker(h)$, and we shall show that the injectivity of the assembly map $\mu$ implies the relative Novikov conjecture.

In \cite{chang-weinb-yu}, Chang, Weinberger and the fourth author defined a relative higher index
$$\mbox{Ind}(D_M, D_{\partial M}) \in KO_*(C^*_{max}(G,\Gamma))$$
for Dirac operators $(D_M, D_{\partial M})$ on spin manifolds with boundary. Here the group $KO_*(C^*_{max}(G,\Gamma))$ is the $KO$-theory of the maximal relative group $C^*$-algebra associated with the group homomorphism $h\colon G\to \Gamma$. They applied their relative higher index to detect the existence of positive scalar curvature metrics on manifolds with boundary and furthermore non-compact manifolds.

We shall define an analogous relative higher index for signature operators on manifolds with boundary.  The same construction below can be used to define both the maximal relative higher index and the  reduced relative higher index (with coefficients in a ${\rm II}_1$-factor $\mathcal M$) for signature operators. For simplicity, we shall only work out the details for the maximal relative higher index.

  We only carry out the details of the even dimensional case; the odd case is similar. Assume that $M$ is an even dimensional manifold with boundary and $D$ is the signature operator on $M$. Define $$M_\infty=M\bigcup\limits_{\partial M}(\partial M\times [0,\infty))$$
to be the manifold  obtained by attaching an infinity cylinder to $M$.
Let $D_\infty$ be  the signature operator on $M_\infty$. Let $\widetilde{M}_\infty$ be the universal cover of $M_{\infty}$ and $\widetilde{D}_\infty$ the lifting of $D_\infty$ on  $\widetilde{M}_\infty$. Since $\widetilde M_\infty$ is a complete manifold (without boundary), a standard  construction of higher indices for elliptic operators on complete manifolds  (cf. \cite{Willett-Yu-higher-index-book}) gives the higher index
 $${\rm Ind}_{max}(\widetilde{D}_\infty)=[p] \in K_0(C_{max}^*(\widetilde{M}_\infty)^\Gamma),$$
where $p$ is an idempotent in the matrix algebra of the unitization of $ C^*(\widetilde{M}_\infty)^{\Gamma}$. By the definition of the higher index (see \cite{Willett-Yu-higher-index-book}), we can choose $p$ so that its propagation is as small as we want.

Denote $\widetilde{M}$ to be the universal covering space of $M$ and we can view $\widetilde{M}$ as a subspace of $\widetilde{M}_{\infty}$. Let $\chi$ be the characteristic function of the subspace $\widetilde{M}$ of $\widetilde{M}_{\infty}$.
Consider the invertible element
$$u=e^{2\pi i(\chi p \chi)}\in \left(C_{max}^*(\widetilde{M})^\Gamma\right)^+\cong \left(C_{max}^*(\Gamma)\otimes \mathcal{K}(H)\right)^+.$$
Denote by $(\partial M)_{\Gamma}$ the space of the restriction of the covering space  $\widetilde{M}$ on $\partial M\subseteq M$. The space $(\partial M)_{\Gamma}$ is a manifold equipped with a proper and cocompact $\Gamma$-action. Let
$$[u]\in K_*(C^*_{max}(\Gamma))$$
be the higher index of the signature operator $\widetilde{D}_{\partial M}$ on $(\partial M)_{\Gamma}$.

Note that there is an integer $n_0$ such that
$$\left|\exp(2\pi ix)-\sum^{n_0}_{k=0}\frac{(2\pi ix)^k}{k!}\right|\leq \frac{1}{1000},$$
for all $x \in \mathbb{R}$. Denote
\begin{equation}\label{eq:approx}
\varphi(x)=\sum^{n_0}_{k=0}\frac{(2\pi ix)^k}{k!}.
\end{equation}
We define $v=\varphi(2\pi i(\chi p \chi))$ in the matrix algebra of $ \left(C_{max}^*(\widetilde{M})^\Gamma\right)^+$.

Since we can choose  $p$ so that its propagation is arbitrarily small, the operator $\chi p \chi$ is an idempotent away from a small tubular neighborhood of $(\partial M)_\Gamma$. By a standard  finite propagation argument, we have that  away from a small tubular neighborhood of $(\partial M)_\Gamma$, the operator $v$ is very close (in operator norm) to $1$.
Consequently,  $v$ restricts to an invertible element in the matrix algebra of $\left(C_{max}^*((\partial M)_\Gamma\times (-\epsilon,0))^\Gamma\right)^+\cong \left(C_{max}^*(\Gamma)\otimes \mathcal{K}(H)\right)^+$ for some $\epsilon >0$. Here the constant $\epsilon$ depends on the propagation of $p$.

 We shall repeat the above construction on  the complete manifold $\partial M\times \mathbb{R}$.  Note that the product space $(\partial M)_{\Gamma} \times \mathbb{R}$ is a $\Gamma$-covering space  of  $\partial M \times \mathbb{R}$. Denote by $D_{(\partial M )_{\Gamma}\times \mathbb{R}}$ the signature operator on  $(\partial M )_{\Gamma} \times \mathbb{R}$. We denote the higher index of $D_{(\partial M )_{\Gamma}\times \mathbb{R}}$ by
 $$[p']={\rm Ind}_{max}(D_{(\partial M)_{\Gamma}\times \mathbb{R}})\in K_0(C_{max}^*((\partial M )_{\Gamma}\times \mathbb{R})^\Gamma),$$
 where $p'$ is an idempotent in the matrix algebra of $\left(C_{max}^*((\partial M )_{\Gamma}\times \mathbb{R})^\Gamma\right)^+$ with small  propagation. Let $\chi'$ be the characteristic function of the subspace $(\partial M )_{\Gamma}\times (-\infty, 0]$ in $(\partial M)_{\Gamma}\times \mathbb{R}$. Define
 $$u'=e^{2\pi i(\chi' p' \chi')}.$$
 By a similar argument as above,  we have that
\begin{enumerate}
   \item[(1)] $v'=\varphi(2\pi i(\chi' p' \chi'))$ is invertible in the matrix algebra of $\left(C_{max}^*( (\partial M)_{\Gamma}\times (-\infty, 0] )^\Gamma\right)^+$;

     \item[(2)] away from a small tubular neighborhood of $(\partial M)_{\Gamma}\times\{0\}$, the operator $v'$ is close (in operator norm) to $1$.
\end{enumerate}
Similar as before, the element $v'$ restricts to an invertible element in $$\left(C_{max}^*({(\partial M)_{\Gamma}\times (-\epsilon,0)})^ \Gamma\right)^+\cong \left(C_{max}^*(\Gamma)\otimes \mathcal{K}(H)\right)^+.$$
By construction,  we have
$$v'=v.$$

Now we also repeat the above construction for the covering space $\widetilde{\partial M}\times \mathbb{R} \rightarrow  \partial M \times \mathbb{R}$ where $\widetilde{\partial M}$ is the universal covering space of $\partial M$. Denote $D_{\widetilde{\partial M}\times \mathbb{R}}$ the signature operator on $\widetilde{\partial M}\times \mathbb{R}$. Let
$$[p'']={\rm Ind}_{max}(D_{\widetilde{\partial M}\times \mathbb{R}})\in K_0(C_{max}^*( \widetilde{\partial M}\times \mathbb{R})^G)$$
be  the higher index of the signature operator $D_{\widetilde{\partial M}\times \mathbb{R}}$ on $\widetilde{\partial M}\times \mathbb{R}$.
Let $\chi''$ is the characteristic function of the subspace $\widetilde{\partial M}\times (-\infty, 0]$ in  $\widetilde{\partial M}\times \mathbb{R}$. Define
\begin{equation}\label{eq:invert}
v''=e^{2\pi i(\chi'' p'' \chi'')},
\end{equation}
which by the same argument above restricts  to  an invertible element in the matrix algebra of
 $$\left(C_{max}^*(\widetilde{\partial M}\times [-\epsilon,0])^G\right)^+\cong \left(C_{max}^*(G)\otimes \mathcal{K}(H)\right)^+.$$
By the functoriality of the higher index, we have that
$$h_{max}(v'')=v',$$
where $h_{max}\colon  C_{max}^*(\widetilde{\partial M}\times [-\epsilon, 0] )^G \rightarrow C_{max}^*(\widetilde{\partial M}\times [-\epsilon, 0] )^\Gamma$ is induced by the natural maps as follows:
 \begin{equation*}
 \begin{tikzcd}
\widetilde{\partial M}\times \mathbb{R} \ar{r}{h}\ar{d}& (\partial M)_{\Gamma}\times \mathbb{R}\ar{d}\\
\partial M \times \mathbb{R}\ar{r}&  \partial M \times \mathbb{R}.
 \end{tikzcd}
 \end{equation*}
Consider the path of invertibles
$$v(t)=\varphi(2\pi i(t\chi p \chi)), t\in [0, 1].$$
 We obtain a path joining $v$ and $\lambda$ in the matrix algebra of $(C_{max}^*(\widetilde{M})^\Gamma)^+$, where $\lambda$ is a constant very close to $1$. By connecting $\lambda$ to $1$ by the linear path $(1-s)\lambda + s $,  we have the following result.
\begin{lem}\label{lm:invpath}
 $v$ is homotopic to $1$ through a continuous path of invertible elements in the matrix algebra of $(C_{max}^*(\widetilde{M})^\Gamma)^+$.
\end{lem}

Now, we are ready to define the relative higher index for the signature operators $(D_M, D_{\partial M})$ on the pair $(M, \partial M)$. Let $h_{max}: C_{max}^*(G)\to C_{max}^*(\Gamma)$ be the homomorphism induced by the homomorphism $h: G \to \Gamma$. Denote by $C_{h_{max}\otimes id}$ the suspension of the cone associated with the homomorphism
$$h_{max} \otimes id: C_{max}^*(G)\otimes \mathcal{K}(H)\rightarrow C_{max}^*(\Gamma)\otimes \mathcal{K}(H).$$
Clearly, we have
$$C_{h_{max} \otimes id}\cong C^*_{max}(G, \Gamma),$$
where $C^*_{max}(G, \Gamma)$ is the suspension of the cone of the homomorphism
$$h_{max}: C^*_{max}(G) \to C^*_{max}(\Gamma).$$
\begin{defn}\label{def:rel}
Let $M$ be an even-dimensional compact oriented manifold with boundary $\partial M$. Define the relative higher index of the signature operators $(D_M, D_{\partial M})$ on $(M,\partial M)$ to be
$${\rm Ind}_{max}(D_M, D_{\partial M})=[(v'',f)]\in K_0(C^*_{max}(G, \Gamma)),$$
where $v''$ is defined in line \eqref{eq:invert} and $f$ is the continuous path of invertible elements given by Lemma \ref{lm:invpath}.
\end{defn}
When the boundary $\partial M$ is empty, this relative higher index is precisely the higher index of the signature operator on $M$.

We can consider the function $\varphi_s (t)=\varphi (t/(s+1))$ on $\mathbb{R}$ for each $s\geq 0$. Replacing the function $\varphi$ with $\varphi_s$ in the definition of relative higher index, we obtain a continuous path of representatives of ${\rm Ind}_{max}(D_M, D_{\partial M})$ whose propagations approach $0$ as $s \to \infty$. This continuous path of representatives defines the local relative index
$$
[D_M, D_{\partial M}] \in K_*^{G, \Gamma}(\underline{E}G, \underline{E}\Gamma),
$$
such that
$$\mu_{max}([D_M, D_{\partial M}])={\rm Ind}_{max}(D_M, D_{\partial M}),$$
where $\mu_{max}: K^{G, \Gamma}(\underline{E}G, \underline{E}\Gamma)\to K_*(C^*_{max}(G, \Gamma))$ is the maximal relative assembly map.

We show that  the relative higher index is invariant under orientation-preserving homotopy equivalences of pairs.
\begin{thm}\label{thm-homotopy-invariant-rel-ind}
The higher index ${\rm Ind}_{max}(D_M, D_{\partial M})$ is a homotopy invariant, that is, if
$\phi: (M,\partial M)\rightarrow (N,\partial N)$
is an orientation-preserving homotopy equivalence between two compact oriented manifolds with boundary, then
$$\phi_*({\rm Ind}_{max}(D_M,D_{\partial M}))={\rm Ind}_{max}(D_N, D_{\partial N})\in K_\ast(C_{max}^*(G,\Gamma))$$
and
$$
\phi_*({\rm Ind}_{red}(D_M, D_{\partial M}))={\rm Ind}_{red}(D_N, D_{\partial N}) \in K_*(C^*_{red}(G, \Gamma, \mathcal  M)).
$$
\end{thm}
\begin{proof}
The same proof below works for both the maximal relative higher index and the  reduced relative higher index (with coefficients in a ${\rm II}_1$-factor $\mathcal M$). For simplicity, we shall only give the details for the maximal case. Also, we shall only prove the even dimensional case. The odd dimensional case is completely similar.

    Let us write
$${\rm Ind}_{max}(D_M, D_{\partial M})=[v''_0, f_0]\in K_0(C^*_{max}(G, \Gamma))$$
and
$${\rm Ind}_{max}(D_N, D_{\partial N})=[v''_1, f_1]\in K_0(C^*_{max}(G, \Gamma)),$$
where $v''_i$ and $f_i$ are given as in Definition \ref{def:rel}.
We shall show that $(v''_0, f_0)$ is homotopic to $(v''_1, f_1)$ in the matrix algebra of $\left(C^*_{max}(G, \Gamma)\otimes \mathcal{K}\right)^+$.

The homotopy equivalence $\phi:M \to N$ induces an equivariant homotopy equivalence
$$\widetilde{\phi}\colon  \widetilde{M} \rightarrow \widetilde{N}.$$
 Denote by $D_{M_{\infty}}$ and $D_{N_{\infty}}$ the signature operators on $M_{\infty}$ and $N_{\infty}$, respectively. Following the arguments in \cite[Section 8]{Xie-Yu-Higher-invariant-in-NCG}, %(cf.  \cite[Theorem 4.3]{Higson-Roe-mapping-Surgery})
we have a continuous path of idempotent $\left( p_t \right)_{t \in [0, 1]}$ in the matrix algebra of $\left(C_{max}^*(\widetilde{N})^\Gamma\right)^+$ connecting $\phi_\ast(p_0)$ and $p_1$, where $[p_0]={\rm Ind}_{max}(D_{M_\infty})$  and $[p_1]={\rm Ind}_{max}(D_{N_\infty})$.
Define a path of invertibles
  $$v_t=\varphi(2\pi i(\chi p_t \chi))$$
in the matrix algebra of $\left(C_{max}^*(\widetilde{N})^\Gamma\right)^+\cong \left(C_{max}^*(\Gamma)\otimes K\right)^+, $
where $\chi$ be the characteristic function of the subspace $\widetilde{N}$ in $\widetilde{N}_\infty$ and $\varphi$ is the function given in line \eqref{eq:approx}.

The homotopy equivalence $\phi: (M, \partial M) \to (N, \partial N)$ also induces a $\Gamma$-equivariant homotopy equivalence
$${\phi_{\partial}^\Gamma}\times id: (\partial M)_{\Gamma} \times \mathbb{R}\rightarrow ( \partial N)_{\Gamma} \times \mathbb{R},$$
where $(\partial M)_{\Gamma}$ (resp. $(\partial N)_{\Gamma}$) is the restriction of the covering space $\widetilde M \to M$ (resp. $\widetilde N\to N$)  over $\partial M$ (resp. $\partial N$).
Let $D_{(\partial M)_{\Gamma}\times \mathbb{R}}$ and $D_{( \partial N)_{\Gamma} \times \mathbb{R}}$ be the signature operator on $( \partial M)_{\Gamma} \times \mathbb{R}$ and $( \partial N)_{\Gamma} \times \mathbb{R}$ respectively. Similarly (cf. \cite[Section 8]{Xie-Yu-Higher-invariant-in-NCG} %\cite[Theorem 4.3]{Higson-Roe-mapping-Surgery}
), there exists a continuous path of idempotents  $\left(p'_t\right)_{t \in[0. 1]}$ connecting $({\phi}_\partial^\Gamma)_\ast(p'_0)$ and $p'_1$, where $[p'_0]={\rm Ind}_{max}(D_{( \partial M)_{\Gamma} \times \mathbb{R}})$ and $[p'_1]= {\rm Ind}_{max}(D_{( \partial N)_{\Gamma}\times \mathbb{R}})$. As a result, we obtain a path of invertibles
$$v'_t=\varphi(2\pi i(\chi' p'_t \chi')),$$
for all $ t\in [0,1]$ in the matrix algebra of $(C_{max}^*(( \partial N)_{\Gamma}\times [-\epsilon, 0])^\Gamma)^+ \cong  (C_{max}^*(\Gamma)\otimes K)^+$ for some  constant $\epsilon>0$, where $\chi'$ be the characteristic function of the subspace $( \partial N)_{\Gamma}\times (-\infty, 0]$ in $( \partial N)_{\Gamma}\times \mathbb{R}$.
Since the path $p_t$ and $p'_t$ are constructed using the same formula,  we have
by construction that $$v'_t=v_t,$$
for all $t \in [0, 1]$.

The homotopy equivalence $\phi:(M, \partial M)\to (N, \partial N)$ also induces a $G$-equivariant homotopy equivalence
$$\widetilde \phi_\partial \colon  \widetilde{\partial M}\times \mathbb{R}\longrightarrow \widetilde{\partial N}\times \mathbb{R},$$
where $\widetilde{\partial M}$ (resp. $\widetilde{\partial N}$) is the universal covering space of $\partial M$ (resp. $\partial N$). Similarly, there is a continuous path of idempotents $\left(p''_t\right)_{t \in [0, 1]}$ in the matrix algebra of $C_{max}^*(\widetilde{\partial N}\times \mathbb R )^G$ connecting $(\widetilde \phi_\partial)_\ast(p''_0)$ and $p''_1$, where $[p''_0]={\rm Ind}_{max}(D_{\widetilde{\partial M}\times \mathbb{R}})$ and $[p_1'']={\rm Ind}_{max}(D_{\widetilde{\partial N}\times \mathbb{R}})$. Hence we obtain a path of invertibles
$$v''_t=\varphi(2\pi i (\chi'' p''_t \chi''))$$
in the matrix algebra of $C_{max}^*(\widetilde{\partial M}\times (-\epsilon, 0] )^G\cong  C_{max}^*(G)\otimes K$ for some constant $\epsilon >0$, where  $\chi''$ be the characteristic function of the subspace $ \widetilde{\partial N}\times (-\infty, 0]$ in $\widetilde{\partial N}\times \mathbb{R}$.

By the definitions of $v'_t$ and $v''_t$, we have that
$$h_{max}(v''_t)=v'_t,$$
for all $ t\in [0, 1]$, where $h_{max}: C^*_{max}(G) \to C^*_{max}(\Gamma)$ is the $*$-homomorphism induced by $h: G = \pi_1(\partial N) \to \Gamma =\pi_1(N)$.

For each fixed $t\in [0, 1]$, let $\{f_t(s)\}_{s\in [0, 1]}$ be the path constructed out of $v_t$ as in Lemma \ref{lm:invpath}. In particular,
$f_t(0)=v_t$ and $f_t(1) = 1$ for all $t\in [0,1]$. Consequently, we have a continuous path $\{(v''_t, f_t)\}_{t\in[0,1]}$ connecting $(v''_0, f_0)$ and $(v''_1, f_1)$.   As a result, we have that
$$\phi_*({\rm Ind}_{max}(D_M,D_{\partial M}))={\rm Ind}_{max}(D_N,D_{\partial N}).$$
\end{proof}

At the end of this section, let us show that the  maximal (reduced) strong relative Novikov conjecture together with Theorem \ref{thm-homotopy-invariant-rel-ind} implies the relative Novikov conjecture regarding the homotopy invariance of relative higher signatures of manifolds with boundary. Let us focus only on the maximal case, since the reduced case is completely analogous.

Let $(M, \partial M)$ be a compact oriented manifold with boundary. Let $\psi_M\colon M \to B\Gamma$ (resp. $\psi_{\partial M}\colon  \partial M \to BG$) be the classifying map associated with the universal covering space. Let $\mathcal L_M$ be the $ L$-class of $M$ and $\mathcal L_{\partial M}$ the $L$-class of the boundary $\partial M$, respectively.
For each element $(\xi,\eta) \in H^*(BG, B\Gamma)$, one can define a relative index pairing as follows:
\begin{equation}\label{eq:pair}
\langle (D_M, D_{\partial M}), (\xi, \eta)\rangle = \int_M \mathcal L_M \cup \psi_M^*(\xi) - \int_{\partial M} \mathcal L_{\partial M} \cup \psi_{\partial M}^*(\eta),
\end{equation}
where $H^*(BG, B\Gamma)$ is the relative group cohomology for the group homomorphism $h: G\to \Gamma$ and $(D_M, D_{\partial M})$ is the signature operator of $(M, \partial M)$.  The right hand side of the above equation is usually referred to as a relative higher signature of $(M, \partial M)$. We refer the reader to \cite{Lesch-Moscovici-Pflaum_Connes-Chern-charactor} for more details on relative index pairings.

\begin{conj}[Relative Novikov conjecture]
All relative higher signatures are invariant under orientation-preserving homotopy equivalences of pairs. More precisely, assume that $\phi\colon (M, \partial M)\to (N, \partial N)$ is an orientation-preserving homotopy equivalence of pairs, then
\begin{equation}\label{eq:equal}
\int_M \mathcal L_M \cup \psi_M^*(\xi) - \int_{\partial M} \mathcal L_{\partial M} \cup \psi_{\partial M}^*(\eta)=
\int_N \mathcal L_N \cup
\psi_N^*(\xi) - \int_{\partial N} \mathcal L_{\partial N} \cup \psi_{\partial N}^*(\eta)
\end{equation}
for all $(\xi,\eta) \in H^*(BG, B\Gamma)$, where $\psi_N\colon N \to B\Gamma$ and $\psi_{\partial N}\colon \partial N \to BG$ are continuous maps such that the following diagram commutes:
\[ \begin{tikzcd}
			(M, \partial M) \ar{rr}{\phi}\ar[dr, "{(\psi_M, \psi_{\partial M})}"'] &  & (N, \partial N) \ar[dl, "{(\psi_N, \psi_{\partial N})}"]\\
			&(B\Gamma,BG) &
	\end{tikzcd}
\]
\end{conj}

There is a relative Connes--Chern character map
$$\textup{Ch}: K_*(BG, B\Gamma)\otimes \mathbb{C} \to H_*(BG, B\Gamma)\otimes \mathbb{C}.$$
It is known that the relative Connes--Chern character $\textup{Ch}$ is an isomorphism, cf.  \cite{Baum-Conn_chern-character}. The pairing in line \eqref{eq:pair} can be viewed as the natural pairing
 \[ H_*(BG, B\Gamma)\otimes \mathbb{C}\times H^\ast(BG, B\Gamma)\otimes \mathbb{C}  \to \mathbb C.\]
It follows that,  if the two $K$-homology classes $\phi_\ast([D_M, D_{\partial M}])$ and $[D_N, D_{\partial N}])$ coincide in $K_*(BG, B\Gamma)\otimes \mathbb{C}$, then the equality \eqref{eq:equal} holds, hence proves the relative Novikov conjecture in this case. However, by Theorem \ref{thm-homotopy-invariant-rel-ind}, we have that
$$\phi_*({\rm Ind}_{max}(D_M,D_{\partial M}))={\rm Ind}_{max}(D_N, D_{\partial N})\in K_\ast(C_{max}^*(G,\Gamma)).$$ Now if the maximal strong relative Novikov conjecture (Conjecture \ref{conj:max}) holds for $h\colon G\to \Gamma$, that is,
\[ \mu_{max}\colon K_*^{G, \Gamma}(\underline{E}G,\underline{E}\Gamma) \to K_*(C^*_{max}(G,\Gamma)) \]
is injective, then it follows that
\[ K_*(BG, B\Gamma)\otimes \mathbb C\to  K_*^{G, \Gamma}(\underline{E}G,\underline{E}\Gamma)\otimes \mathbb C \xrightarrow{\ \mu_{max}\ } K_*(C^*_{max}(G,\Gamma))\otimes \mathbb C  \]
is injective, since the natural homomorphism $K_*(BG, B\Gamma)\otimes \mathbb C\to  K_*^{G, \Gamma}(\underline{E}G,\underline{E}\Gamma)\otimes \mathbb C$ is always injective. Since we have
\[ \mu_{max}(\phi_\ast([D_M, D_{\partial M}]) ) =  \phi_*({\rm Ind}_{max}(D_M,D_{\partial M})) \]
and
\[ \mu_{max}([D_N, D_{\partial N}]) =  {\rm Ind}_{max}(D_N,D_{\partial N}), \]
it follows that
\[ \phi_\ast([D_M, D_{\partial M}]) = [D_N, D_{\partial N}] \in K_*(BG, B\Gamma)\otimes \mathbb{C}.
\]
Therefore, this shows that the maximal strong relative Novikov conjecture implies the relative Novikov conjecture. The implication that the relative Novikov conjecture follows from the reduced strong relative Novikov conjecture is similar. We omit the details.

\begin{rmk}
In \cite{Lei-Lott-Piaz}, Leichtnam--Lott--Piazza defined a higher index for signature operators on manifolds with boundary, under certain invertibility assumptions of the signature operator on the boundary. We point out that the relative higher index of signature operators defined in the current paper is generally different from  the higher index of Leichtnam--Lott--Piazza. While the construction of the higher index by Leichtnam, Lott and Piazza requires an invertibility condition of the signature operator on the boundary, the relative higher index in our paper is always defined without any invertiblity condition on the boundary. On the other hand, the higher index of Leichtnam, Lott and Piazza lies in $K_*(C^*_{red}(\Gamma))$ instead of the $K$-theory of the relative group $C^\ast$-algebra, due to the extra invertibility condition on the boundary. The two (relative) higher indices are related as follows. Let us assume the invertibility condition on the boundary as in \cite{Lei-Lott-Piaz} so that the higher index of  Leichtnam--Lott--Piazza is defined. Then the image of Leichtnam--Lott--Piazza's higher index under the boundary map $K_*(C^*_{red}(\Gamma,\mathcal{M}\otimes \mathcal{M}))\to K_*(C^*_{red}(G,\Gamma, \mathcal{M}))$ coincides with our relative higher index for the pair $(M,\partial M)$.
\end{rmk}

\nocite{Weinberger-Xie-Yu-additivity-of-higher-rho, Xie-Yu-Higher-invariant-in-NCG, MR3928083}
\nocite{Tang-Jerome}
\bibliographystyle{plain}
\bibliography{ref}

\end{document}